\documentclass[a4paper]{article}
\usepackage{geometry}
\usepackage{amssymb}
\usepackage{latexsym}
\usepackage{amsmath}
\usepackage{indentfirst}
\usepackage{graphicx}
 \usepackage[colorlinks=true]{hyperref}
\usepackage{mathrsfs} 
\usepackage{chngcntr}
\usepackage{color}

\newtheorem{Thm}{Theorem}[section]

\newtheorem{Lem}{Lemma}[section]
\newtheorem{Pro}{Proposition}[section]

\newtheorem{Rem}{Remark}[section]

\numberwithin{equation}{section}

\newenvironment{proof}{\medskip\par\noindent{\bf Proof\/}:\quad}{\qquad
\raisebox{-0.5mm}{\rule{1.5mm}{1mm}}\vspace{6pt}}
\begin{document}
\title{Local uniqueness of semiclassical bounded states for a singularly perturbed fractional Kirchhoff problem}

\author{
{Vicen\c tiu D. R\u{a}dulescu}\thanks{Email: {\tt radulescu@inf.ucv.ro}}\\
\small Faculty of Applied Mathematics, AGH University of Science and Technology, 30-059 Krak\'{o}w, Poland\\
\small Department of Mathematics, University of Craiova, Craiova, 200585, Romania\\
{Zhipeng Yang}\thanks{Corresponding author. Email: {\tt yangzhipeng326@163.com}}\\
\small Mathematical Institute, Georg-August-University of G\"ottingen, G\"ottingen 37073, Germany\\
}

\date{} \maketitle

\textbf{Abstract.}
In this paper, we consider the following singularly perturbed fractional Kirchhoff problem
\begin{equation*}
\Big(\varepsilon^{2s}a+\varepsilon^{4s-N} b{\int_{\mathbb{R}^{N}}}|(-\Delta)^{\frac{s}{2}}u|^2dx\Big)(-\Delta)^su+V(x)u=|u|^{p-2}u,\quad \text{in}\ \mathbb{R}^{N},
\end{equation*}
where $a,b>0$, $2s<N<4s$ with $s\in(0,1)$, $2<p<2^*_s=\frac{2N}{N-2s}$ and $(-\Delta )^s$ is the fractional Laplacian. For $\varepsilon> 0$ sufficiently small and a bounded continuous function $V$, we establish a type of local Pohoz\v{a}ev identity by extension technique and then we can obtain the local uniqueness of semiclassical bounded solutions based on our recent results on the uniqueness and non-degeneracy of positive solutions to the limit problem.

\vspace{6mm} \noindent{\bf Keywords:} Fractional Kirchhoff equations; Pohoz\v{a}ev identity; Lyapunov-Schmidt reduction.

\vspace{6mm} \noindent


\section{Introduction and main results}
Let $H^s(\mathbb{R}^N) (0<s<1)$ be the fractional Sobolev space defined by
\begin{equation*}
H^{s}(\mathbb{R}^N)=\bigg\{u\in L^2(\mathbb{R}^N):\frac{u(x)-u(y)}{|x-y|^{\frac{N}{2}+s}}\in L^2({\mathbb{R}^{N}\times\mathbb{R}^{N}})\bigg\},
\end{equation*}
endowed with the natural norm
\begin{equation*}
  \|u\|^2=\int_{\mathbb{R}^N}|u|^2dx+\int\int_{\mathbb{R}^N\times\mathbb{R}^N}\frac{|u(x)-u(y)|^2}{|x-y|^{N+2s}}dxdy.
\end{equation*}
We continue to consider the following singularly perturbed fractional Kirchhoff problem
\begin{equation}\label{eq1.1}
\Big(\varepsilon^{2s}a+\varepsilon^{4s-N} b{\int_{\mathbb{R}^{N}}}|(-\Delta)^{\frac{s}{2}}u|^2dx\Big)(-\Delta)^su+V(x)u=u^p,\quad \text{in}\ \mathbb{R}^{N},
\end{equation}
where $a,b>0$, $\varepsilon>0$ is a parameter, $V: \mathbb{R}^{N} \rightarrow \mathbb{R}$ is a bounded continuous function and $p$ satisfies
\begin{equation*}
1<p<2_s^*-1=
\begin{cases}
\frac{N+2s}{N-2s}, &  0<s<\frac{N}{2}, \\
+\infty, &  s\geq \frac{N}{2},
\end{cases}
\end{equation*}
where $2^*_s$ is the standard fractional Sobolev critical exponent. The fractional Laplacian $(-\Delta)^s$ is the pseudo-differential
operator defined by
\begin{equation*}
\mathcal{F}((-\Delta)^su)(\xi)=|\xi|^{2s}\mathcal{F}(u)(\xi),\ \ \xi\in \mathbb{R}^N,
\end{equation*}
where $\mathcal{F}$ denotes the Fourier transform. It is also given by
\begin{equation*}
(-\Delta)^{s}u(x)=-\frac{1}{2}C(N,s)\int_{\mathbb{R}^N}\frac{u(x+y)+u(x-y)-2u(x)}{|y|^{N+2s}}dy,
\end{equation*}
where
\begin{equation*}
C(N,s)=\bigg(\int_{\mathbb{R}^N}\frac{(1-cos\xi_1)}{|\xi|^{N+2s}}d\xi\bigg)^{-1},\ \xi=(\xi_1,\xi_2,...\xi_N).
\end{equation*}
\par
If $s=1$, Eq.\eqref{eq1.1} reduces to the well known Kirchhoff type problem, which and their variants have been studied extensively in the literature. The equation that goes under the name of Kirchhoff equation was proposed in \cite{Kirchhoff1883} as a model for the transverse oscillation of a stretched string in the form
\begin{equation}\label{eq1.2}
\rho h \partial_{t t}^{2} u-\left(p_{0}+\frac{\mathcal{E}h}{2 L} \int_{0}^{L}\left|\partial_{x} u\right|^{2} d x\right) \partial_{x x}^{2} u=0,
\end{equation}
for $t \geq 0$ and $0<x<L$, where $u=u(t, x)$ is the lateral displacement at time $t$ and at position $x, \mathcal{E}$ is the Young modulus, $\rho$ is the mass density, $h$ is the cross section area, $L$ the length of the string, $p_{0}$ is the initial stress tension.
\par
Problem \eqref{eq1.2} and its variants have been studied extensively in the literature. Bernstein obtains the global stability result in \cite{Bernstein1940BASUS}, which has been generalized to arbitrary dimension $N\geq 1$ by Poho\v{z}aev in \cite{Pohozaev1975MS}. We also point out that such problems may describe a process of some biological systems dependent on the average of itself, such as the density of population (see e.g. \cite{Arosio-Panizzi1996TAMS}).
After Lions \cite{MR519648} introducing an abstract functional framework to this problem, this type of problem has received
much attention.
We refer to e.g. \cite{Chen-Kuo-Wu2011JDE,Perera-Zhang2006JDE,Shuai2015JDE} and to e.g. \cite{Deng-Peng-Shuai2015JFA,MR3218834,MR3360660,MR4293909,MR3412403,MR4305432,MR4021897,MR3200382} for mathematical researches on Kirchhoff type equations on bounded domains and in the whole space, respectively.
We also refer to \cite{MR3987384} for a recent survey of the results connected to this model.
\par
On the other hand, the interest in generalizing the model introduced by Kirchhoff to the fractional case does
not arise only for mathematical purposes. In fact, following the ideas of \cite{MR2675483} and the concept
of fractional perimeter, Fiscella and Valdinoci proposed in \cite{MR3120682} an equation describing the
behaviour of a string constrained at the extrema in which appears the fractional length of the rope. Recently, problem similar to \eqref{eq1.1} has been extensively investigated by many authors using different techniques and producing several relevant results  (see, e.g. \cite{MR4071927,MR4056169,MR4241293,MR4305431,MR4245633,MR3985380,Gu-Yang,MR3877201,MR3961733,MR3917341,MR3993416}).
\par
From the viewpoint of calculus of variation, the fractional Kirchhoff problem \eqref{eq1.1} is much more complex and difficult than the
classical fractional Laplacian equation as the appearance of the term $b\big({\int_{\mathbb{R}^{N}}}|(-\Delta)^{\frac{s}{2}}u|^2dx\big)(-\Delta )^su$, which is of order four.
This fact leads to difficulty in obtaining the boundness of the $(PS)$ sequence for the corresponding energy functional if $p\leq3$.
Recently, R\u{a}dulescu and Yang \cite{R-Yang} established uniqueness and nondegeneracy for positive solutions to Kirchhoff equations with subcritical growth. More precisely,
they proved that the following fractional Kirchhoff equation
\begin{equation}\label{eq1.3}
\Big(a+b{\int_{\mathbb{R}^{N}}}|(-\Delta)^{\frac{s}{2}}u|^2dx\Big)(-\Delta)^su+mu=u^p,\quad \text{in}\ \mathbb{R}^{N},
\end{equation}
where $a,b,m>0$, $\frac{N}{4}<s<1$, $1<p<2^*_s-1$, has a unique nondegenerate positive radial solution.  One of the main idea is based on the scaling technique which allows us to find a relation between solutions of  \eqref{eq1.3} and the following equation
\begin{equation}\label{eq1.4}
(-\Delta)^{s} Q+Q=Q^{p} , \quad \text { in } \mathbb{R}^{N}
\end{equation}
where $0<s<1$ and $1<p<2_{s}^{*}-1$. For high dimension and critical case we refer to \cite{Gu-Yang1,Yang2,Yang,Yang-Yu}.
We first summarize the main results in \cite{R-Yang} for convenience.
\begin{Pro}\label{Pro1.1}Let $a,b,m>0$ Assume that $\frac{N}{4}<s<1$ and $1<p<2_s^*-1$. Then
equation \eqref{eq1.3} has a ground state solution $U\in H^s(\mathbb{R}^N)$,
\begin{itemize}
  \item[$(i)$] $U>0$ belongs to $C^\infty(\mathbb{R}^N)\cap H^{2s+1}(\mathbb{R}^N)$;
  \item[$(ii)$] there exist some $x_0\in \mathbb{R}^N$ such that
  $U(\cdot-x_0)$ is radial and strictly decreasing in $r=|x-x_0|$;
  \item[$(iii)$] there exist constants $C_1,C_2>0$ such that
\begin{equation*}
\frac{C_1}{1+|x|^{N+2s}}\leq U(x)\leq \frac{C_2}{1+|x|^{N+2s}},\quad \forall \ x\in \mathbb{R}^N.
\end{equation*}
  \end{itemize}
\end{Pro}

\begin{Pro}\label{Pro1.2}
Under the assumptions of Proposition \ref{Pro1.1}, the ground state solution $U$ of \eqref{eq1.3} is unique up to
 translation. Moreover, $U$ is nondegenerate in $H^s(\mathbb{R}^N)$ in the sense that there holds
$$\ker L_+=span\{\partial_{x_1}U, \partial_{x_2}U,\cdots,  \partial_{x_N}U\},$$
where $L_+$ is defined as
\begin{equation*}
  L_+\varphi=\Big(a+b{\int_{\mathbb{R}^{N}}}|(-\Delta
)^{\frac{s}{2}}U|^2dx\Big)(-\Delta
)^s\varphi+m\varphi-pU^{p-1}\varphi+2b\Big({\int_{\mathbb{R}^{N}}}(-\Delta
)^{\frac{s}{2}}U(-\Delta )^{\frac{s}{2}}\varphi dx\Big)(-\Delta )^sU
\end{equation*}
acting on $L^2(\mathbb{R}^N)$ with domain $H^s(\mathbb{R}^N)$.
\end{Pro}
\par
By Proposition \ref{Pro1.2}, it is now possible that we apply Lyapunov-Schmidt reduction to study the perturbed fractional Kirchhoff equation \eqref{eq1.1}. We want to look for solutions of \eqref{eq1.1} in the Sobolev space $H^s(\mathbb{R}^N)$ for sufficiently small $\varepsilon$, which named semiclassical solutions. We also call such derived solutions as concentrating solutions since they will concentrate at certain point of the potential function $V$. In fact, this method can be traced back to Floer and Weinstein 's work (\cite{MR867665}) on the Schr\"{o}dinger equation.  Later, Oh \cite{MR970154,MR1065671} generalized Floer–Weinstein’s results to higher dimension and obtained the existence of positive multi-bump solutions concentrating at any given set
of nondegenerate critical points of $V(x)$ as $\varepsilon\rightarrow0$. Also, the existence of a single-peak solution concentrating at the critical point of $V(x)$ which may be degenerate as $\varepsilon\rightarrow0$ was obtained by Ambrosetti et al. \cite{MR1486895}.
\par
Now, it is expected that this approach can deal with problem \eqref{eq1.1} for all $1<p<2^*_s-1$, in a unified way.
To state our following results, we first fix some notations that will be used throughout the paper. For $\varepsilon>0$ and $y=\left(y_{1}, y_{2},\cdots y_{N}\right) \in \mathbb{R}^{N}$, write
\begin{equation*}
U_{\varepsilon, y}(x)=U\left(\frac{x-y}{ \varepsilon}\right), \quad x \in \mathbb{R}^{N}.
\end{equation*}
Assume that $V: \mathbb{R}^{N} \rightarrow \mathbb{R}$ satisfies the following conditions:
\begin{itemize}
  \item [$(V_1)$] $V$ is a bounded continuous function with $\inf\limits_{x \in \mathbb{R}^{N}} V>0$;
  \item [$(V_2)$] There exist $x_{0} \in \mathbb{R}^{N}$ and $r_{0}>0$ such that
\begin{equation*}
V\left(x_{0}\right)<V(x) \quad \text { for } 0<\left|x-x_{0}\right|<r_{0}
\end{equation*}
and $V \in C^{\alpha}\left(\bar{B}_{r_{0}}\left(x_{0}\right)\right)$ for some $0<\alpha<\frac{N+4 s}{2}$. That is, $V$ is of $\alpha$-th order H\"{o}lder continuity around $x_{0}$;
\item [$(V_3)$] There exist $m>1$ and $\delta>0$ such that
\begin{equation*}
\begin{cases}V(x)=V\left(x_{0}\right)+\sum\limits_{i=1}^{N} c_{i}\left|x_{i}-x_{0, i}\right|^{m}+O\left(\left|x-x_{0}\right|^{m+1}\right), & x \in B_{\delta}\left(x_{0}\right) \\ \frac{\partial V}{\partial x_{i}}=m c_{i}\left|x_{i}-x_{0, i}\right|^{m-2}\left(x_{i}-x_{0, i}\right)+O\left(\left|x-x_{0}\right|^{m}\right), & x \in B_{\delta}\left(x_{0}\right)\end{cases}
\end{equation*}
where $c_{i} \in \mathbb{R}$ and $c_{i} \neq 0$ for $i=1,2,\cdots,N$.
\end{itemize}
Without loss of generality, we assume $x_0=0$ for simplicity. The assumption $(V_1)$ allows us to introduce the inner products
\begin{equation*}
\langle u, v\rangle_{\varepsilon}=\int_{\mathbb{R}^{N}}\left(\varepsilon^{2s} a (-\Delta)^{\frac{s}{2}} u \cdot(-\Delta)^{\frac{s}{2}} v+V(x) u v\right)dx
\end{equation*}
for $u, v \in H^{s}\left(\mathbb{R}^{N}\right)$. We also write
\begin{equation*}
H_{\varepsilon}=\left\{u \in H^{s}\left(\mathbb{R}^{N}\right):\|u\|_{\varepsilon}=\langle u, u\rangle_{\varepsilon}^{\frac{1}{2}}<\infty\right\}.
\end{equation*}
Now we state the existence result as follows.
\begin{Thm}\label{Thm1.1}
Under the assumptions of Proposition \ref{Pro1.1} and assume that $V$ satisfies $(V_1)-(V_3)$. Then there exists $\varepsilon_{0}>0$ such that for all $\varepsilon \in\left(0, \varepsilon_{0}\right)$,
\begin{itemize}
  \item [$(i)$] problem \eqref{eq1.1} has a positive solution $u_{\varepsilon}$, which has a global maximum point $\eta_\varepsilon\in\mathbb{R}^N$ such that $\lim\limits_{\varepsilon\rightarrow0}V(\eta_\varepsilon)=V(x_0)$;
  \item  [$(ii)$] If $u_{\varepsilon}^{(i)}, i=1,2$, are two solutions derived as above, then
\begin{equation*}
u_{\varepsilon}^{(1)} \equiv u_{\varepsilon}^{(2)}
\end{equation*}
holds for $\varepsilon$ sufficiently small.
\end{itemize}
Moreover, let
\begin{equation*}
u_{\varepsilon}=U\left(\frac{x-y_{\varepsilon}}{\varepsilon}\right)+\varphi_{\varepsilon}.
\end{equation*}
be the unique solution, then there hold
\begin{equation*}
\begin{gathered}
\left|y_{\varepsilon}\right|=o(\varepsilon), \\
\left\|\varphi_{\varepsilon}\right\|_{\varepsilon}=O\left(\varepsilon^{\frac{N}{2}+m(1-\tau)}\right),
\end{gathered}
\end{equation*}
for some $0<\tau<1$ sufficiently small.
\end{Thm}
\par
As in \cite{R-Yang}, we only need the conditions $(V_1)$ and $(V_2)$ to obtain the existence of the semiclassical solutions by Lyapunov-Schmidt reduction. To prove the local uniqueness, we will follow the idea of \cite{MR3426103}. More precisely, if $u_{\varepsilon}^{(i)}$, $i=1,2$, are two distinct solutions, then it is clear that the function
\begin{equation*}
\xi_{\varepsilon}=\frac{u_{\varepsilon}^{(1)}-u_{\varepsilon}^{(2)} }{\left\|u_{\varepsilon}^{(1)}-u_{\varepsilon}^{(2)}\right\|_{L^{\infty}\left(\mathbb{R}^{N}\right)}}
\end{equation*}
satisfies $\left\|\xi_{\varepsilon}\right\|_{L^{\infty}\left(\mathbb{R}^{N}\right)}=1$. We will show, by using the equations satisfied by $\xi_{\varepsilon}$, that $\left\|\xi_{\varepsilon}\right\|_{L^{\infty}\left(\mathbb{R}^{N}\right)} \rightarrow$ 0 as $\varepsilon \rightarrow 0$. This gives a contradiction, and thus follows the uniqueness. To deduce the contradiction, we will need quite delicate estimates on the asymptotic behaviors of solutions and the concentrating point $y_{\varepsilon}$. Our results extend the results in \cite{MR4021897} to the fractional Kirchhoff problem. But one will see that it is much more difficult to obtain some of the estimates due to the presence of the nonlocal term $\left(\int_{\mathbb{R}^{N}}|(-\Delta)^{\frac{s}{2}} u|^{2}\right)(-\Delta)^s u$ and the fractional operator.

\par
This paper is organized as follows. We prove the existence of semiclassical solutions in Section 2 and consider their concentration behavior in Section 3. In Section 4, we give the local Pohoz\v{a}ev identity and some basic estimates which will be used later. Finally, we finish the proof of local uniqueness in Theorem \ref{Thm1.1}.
\par
\vspace{3mm}
{\bf Notation.~}Throughout this paper, we make use of the following notations.
\begin{itemize}
\item[$\bullet$]  For any $R>0$ and for any $x\in \mathbb{R}^N$, $B_R(x)$ denotes the ball of radius $R$ centered at $x$;
\item[$\bullet$]  $\|\cdot\|_q$ denotes the usual norm of the space $L^q(\mathbb{R}^N),1\leq q\leq\infty$;
\item[$\bullet$]  $o_n(1)$ denotes $o_n(1)\rightarrow 0$ as $n\rightarrow\infty$;
\item[$\bullet$]  $C$ or $C_i(i=1,2,\cdots)$ are some positive constants may change from line to line.
\end{itemize}

\section{Semiclassical solutions for the fractional Kirchhoff equation}
In this section, we mainly prove the existence of semiclassical solution for the fractional Kirchhoff equation \eqref{eq1.1} via Lyapunov-Schmidt reduction method. In fact, all the results in the following has been obtained in \cite{R-Yang} under the condition $(V_1)-(V_2)$. We recall the result and give the outline for the proof for convenience.
\begin{Thm}\label{Thm2.1}
Let $a,b>0$. Assume that $\frac{N}{4}<s<1$ and $1<p<2_s^*-1$, and assume that $V$ satisfies $(V_1)$ and $(V_2)$. Then there exists $\varepsilon_{0}>0$ such that for all $\varepsilon \in\left(0, \varepsilon_{0}\right)$, problem \eqref{eq1.1} has a solution $u_{\varepsilon}$ of the form
\begin{equation*}
u_{\varepsilon}=U\left(\frac{x-y_{\varepsilon}}{\varepsilon}\right)+\varphi_{\varepsilon}
\end{equation*}
with $y_\varepsilon\in\mathbb{R}^N$, $\varphi_{\varepsilon} \in H_{\varepsilon}$, satisfying
\begin{equation*}
\begin{aligned}
y_{\varepsilon} & \rightarrow x_{0}, \\
\left\|\varphi_{\varepsilon}\right\|_{\varepsilon} &=o\left(\varepsilon^{\frac{N}{2}}\right) ,\\
\end{aligned}
\end{equation*}
as $\varepsilon \rightarrow 0$ .
\end{Thm}
\par
It is known that every solution to Eq. \eqref{eq1.1} is a critical point of the energy functional $I_{\varepsilon}: H_{\varepsilon} \rightarrow \mathbb{R}$, given by
\begin{equation*}
I_{\varepsilon}(u)=\frac{1}{2}\|u\|_{\varepsilon}^{2}+\frac{b\varepsilon^{4s-N}}{4}\left(\int_{\mathbb{R}^{N}}|(-\Delta)^{\frac{s}{2}} u|^{2}dx\right)^{2}-\frac{1}{p+1} \int_{\mathbb{R}^{N}} u^{p+1}dx
\end{equation*}
for $u \in H_{\varepsilon}$. It is standard to verify that $I_{\varepsilon} \in C^{2}\left(H_{\varepsilon}\right) .$ So we are left to find a critical point of $I_{\varepsilon}$.  However, due to the presence of the double nonlocal terms $(-\Delta)^s$ and $\left(\int_{\mathbb{R}^{N}}|(-\Delta)^{\frac{s}{2}} u|^{2}\right)$, it requires more careful estimates on the orders of $\varepsilon$ in the procedure. In particular, the nonlocal terms brings new difficulties in the higher order remainder term, which is more complicated than the case of the fractional Schr\"{o}dinger equation or usual Kirchhoff equation.
\par
Denote $Q$ be the solution of the following equation:
\begin{equation*}
\begin{cases}
(-\Delta)^{s} u+u=u^{p}, & \\
u(0)=\max\limits_{x \in \mathbb{R}^{N}} u(x), & \\
u>0, x \in \mathbb{R}^{N}.
\end{cases}
\end{equation*}
Then, it is easy to see that the function
\begin{equation*}
W_{\lambda}(x):=\lambda^{\frac{1}{p}} Q\left(\lambda^{\frac{1}{2 s}} x\right)
\end{equation*}
satisfies the equation
\begin{equation*}
(-\Delta)^{s} u+\lambda u=u^{p}, \quad x \in \mathbb{R}^{N}.
\end{equation*}
Therefore, for any point $\xi \in \mathbb{R}^{N}$, taking $\lambda=V(\xi)$, it follows that $V(\xi)^{\frac{1}{p-1}} W\left(V(\xi)^{\frac{1}{2 s}} x\right)$ satisfies
\begin{equation*}
\varepsilon^{2 s}(-\Delta)^{s} u+V(\xi) u=u^{p}, x \in \mathbb{R}^{N}.
\end{equation*}
By the same idea and the proof of Theorem \ref{Thm1.1}, we have known  $U(x)=m^{\frac{1}{p-1}}Q(m^{\frac{1}{2s}}\mathcal{E}_0^{-\frac{1}{2s}}x)$ is a positive unique solution of \eqref{eq1.3}, where $\mathcal{E}_0=a+b\|(-\Delta)^{\frac{s}{2}}U\|^2_2$. Moreover,
we have the polynomial decay instead of the usual exponential decay of $U$ of and its derivatives (see Section 3). That is,
\begin{equation}\label{eq2.1}
U(x)+|(-\Delta)^{\frac{s}{2}} U(x)| \leq \frac{C}{1+|x|^{N+2s}}, \quad x \in \mathbb{R}^{N}
\end{equation}
for some $C>0$.
\par
Following the idea from Cao and Peng \cite{MR2581983} (see also \cite{MR4021897}) , we will use the unique ground state $U$ of \eqref{eq1.3} with $m=V(x_0)$ to build the solutions of \eqref{eq1.1}. Since the $\varepsilon$-scaling makes it concentrate around $\xi$, this function constitutes a good positive approximate solution to \eqref{eq1.1}.
\par
For $\delta,\eta>0$, fixing $y \in B_{\delta}(x_0)$, we define
\begin{equation*}
M_{\varepsilon,\eta}=\left\{(y, \varphi): y \in B_{\delta}(x_0), \varphi\in E_{\varepsilon,y}, \|\varphi\|^2\leq\eta\varepsilon^{N}\right\}
\end{equation*}
where we denote $E_{\varepsilon,y}$ by
\begin{equation*}
E_{\varepsilon, y}:=\left\{\varphi \in H_{\varepsilon}:\left\langle\frac{\partial U_{\varepsilon, y^{i}}}{\partial y^{i}}, \varphi\right\rangle_{\varepsilon}=0, i=1, \ldots, N\right\}.
\end{equation*}
\par
We will restrict our argument to the existence of a critical point of $I_{\varepsilon}$ that concentrates, as $\varepsilon$ small enough, near the spheres with radii $\frac{r_{0}}{ \varepsilon}$. Thus we
are looking for a critical point of the form
\begin{equation*}
u_\varepsilon=U_{\varepsilon,y}+\varphi_{\varepsilon}
\end{equation*}
where $\varphi_{\varepsilon} \in E_{\varepsilon, y}$, and $\varepsilon y\rightarrow r_{0},\left\|\varphi_{\varepsilon}\right\|^{2}=o\left(\varepsilon^{N}\right)$ as $\varepsilon \rightarrow 0$. For this we introduce a new functional $J_{\varepsilon}:M_{\varepsilon,\eta} \rightarrow \mathbb{R}$ defined by
\begin{equation*}
J_{\varepsilon}(y, \varphi)=I_{\varepsilon}\left(U_{\varepsilon, y}+\varphi\right), \quad \varphi \in E_{\varepsilon, y}.
\end{equation*}
In fact, we divide the proof of first part of Theorem \ref{Thm1.1} into two steps:
\begin{itemize}
  \item [{\bf Step1:}] for each $\varepsilon, \delta$ sufficiently small and for each $y \in B_{\delta}(x_0)$, we will find a critical point $\varphi_{\varepsilon, y}$ for $J_{\varepsilon}(y, \cdot)$ (the function $y \mapsto \varphi_{\varepsilon, y}$ also belongs to the class $C^{1}\left(H_{\varepsilon}\right)$ );
  \item [{\bf Step2:}] for each $\varepsilon, \delta$ sufficiently small, we will find a critical point $y_{\varepsilon}$ for the function $j_{\varepsilon}$ : $B_{\delta}(x_0) \rightarrow \mathbb{R}$ induced by
\begin{equation}\label{eq2.2}
y \mapsto j_{\varepsilon}(y) \equiv J\left(y, \varphi_{\varepsilon, y}\right).
\end{equation}
That is, we will find a critical point $y_{\varepsilon}$ in the interior of $B_{\delta}(x_0)$.
\end{itemize}
\par
It is standard to verify that $\left(y_{\varepsilon}, \varphi_{\varepsilon, y_{\varepsilon}}\right)$ is a critical point of $J_{\varepsilon}$ for $\varepsilon$ sufficiently small by the chain rule. This gives a solution $u_{\varepsilon}= U_{\varepsilon, y_{\varepsilon}}+\varphi_{\varepsilon, y_{\varepsilon}}$ to Eq. \eqref{eq1.1} for $\varepsilon$ sufficiently small in virtue of the following lemma.
\begin{Lem}\label{Lem2.1}
There exist $\varepsilon_{0}, \eta_{0}>0$ such that for $\varepsilon \in\left(0, \varepsilon_{0}\right], \eta \in\left(0, \eta_{0}\right]$, and $(y, \varphi) \in M_{\varepsilon, y}$ the following are equivalent:
\begin{itemize}
  \item [$(i)$] $u_{\varepsilon}= U_{\varepsilon, y_{\varepsilon}}+\varphi_{\varepsilon, y_{\varepsilon}}$ is a critical point of $I_{\varepsilon}$ in $H_{\varepsilon}$.
  \item [$(i)$] $(y, \varphi)$ is a critical point of $J_{\varepsilon}$.
\end{itemize}
\end{Lem}
\par
Now, in order to realize {\bf Step 1}, we expand $J_{\varepsilon}(y, \cdot)$ near $\varphi=0$ for each fixed $y$ as follows:
\begin{equation*}
J_{\varepsilon}(y, \varphi)=J_{\varepsilon}(y, 0)+l_{\varepsilon}(\varphi)+\frac{1}{2}\left\langle\mathcal{L}_{\varepsilon} \varphi, \varphi\right\rangle+R_{\varepsilon}(\varphi)
\end{equation*}
where $J_{\varepsilon}(y, 0)=I_{\varepsilon}\left(U_{\varepsilon, y}\right)$, and $l_{\varepsilon}, \mathcal{L}_{\varepsilon}$ and $R_{\varepsilon}$ are defined for $\varphi, \psi \in H_{\varepsilon}$ as follows:
\begin{equation}\label{eq2.3}
\begin{aligned}
l_{\varepsilon}(\varphi) &=\left\langle I_{\varepsilon}^{\prime}\left(U_{\varepsilon, y}\right), \varphi\right\rangle \\
&=\left\langle U_{\varepsilon, y}, \varphi\right\rangle_{\varepsilon}+b\varepsilon^{4s-N}\left(\int_{\mathbb{R}^{N}}\left| (-\Delta)^{\frac{s}{2}}U_{\varepsilon, y}\right|^{2}dx\right) \int_{\mathbb{R}^{N}} (-\Delta)^{\frac{s}{2}} U_{\varepsilon, y} \cdot (-\Delta)^{\frac{s}{2}} \varphi dx-\int_{\mathbb{R}^{N}} U_{\varepsilon, y}^{p} \varphi dx
\end{aligned}
\end{equation}
and $\mathcal{L}_{\varepsilon}: L^{2}\left(\mathbb{R}^{N}\right) \rightarrow L^{2}\left(\mathbb{R}^{N}\right)$ is the bilinear form around $U_{\varepsilon, y}$ defined by
\begin{equation*}
\begin{aligned}
\left\langle\mathcal{L}_{\varepsilon} \varphi, \psi\right\rangle &=\left\langle I_{\varepsilon}^{\prime \prime}\left(U_{\varepsilon, y}\right)[\varphi], \psi\right\rangle \\
&=\langle\varphi, \psi\rangle_{\varepsilon}+ b\varepsilon^{4s-N}\left(\int_{\mathbb{R}^{N}}\left|(-\Delta)^{\frac{s}{2}} U_{\varepsilon, y}\right|^{2}dx\right) \int_{\mathbb{R}^{N}} (-\Delta)^{\frac{s}{2}} \varphi \cdot (-\Delta)^{\frac{s}{2}} \psi dx\\
&+2\varepsilon^{4s-N}b\left(\int_{\mathbb{R}^{N}} (-\Delta)^{\frac{s}{2}} U_{\varepsilon, y} \cdot (-\Delta)^{\frac{s}{2}} \varphi dx\right)\left(\int_{\mathbb{R}^{N}} (-\Delta)^{\frac{s}{2}} U_{\varepsilon, y} \cdot (-\Delta)^{\frac{s}{2}} \psi dx\right)-p \int_{\mathbb{R}^{N}} U_{\varepsilon, y}^{p-1} \varphi \psi dx
\end{aligned}
\end{equation*}
and $R_{\varepsilon}$ denotes the second order reminder term given by
\begin{equation}\label{eq2.4}
R_{\varepsilon}(\varphi)=J_{\varepsilon}(y, \varphi)-J_{\varepsilon}(y, 0)-l_{\varepsilon}(\varphi)-\frac{1}{2}\left\langle\mathcal{L}_{\varepsilon} \varphi, \varphi\right\rangle.
\end{equation}
We remark that $R_{\varepsilon}$ belongs to $C^{2}\left(H_{\varepsilon}\right)$ since so is every term in the right hand side of \eqref{eq2.4}.
In the rest of this section, we consider $l_{\varepsilon}: H_{\varepsilon} \rightarrow \mathbb{R}$ and $R_{\varepsilon}: H_{\varepsilon} \rightarrow \mathbb{R}$ and give some elementary estimates.
\par
\begin{Lem}\label{Lem2.2}
Assume that $V$ satisfies $(V_1)$ and $(V_2)$. Then, there exists a constant $C>0$, independent of $\varepsilon$, such that for any $y \in B_{1}(0)$, there holds
\begin{equation*}
\left|l_{\varepsilon}(\varphi)\right| \leq C \varepsilon^{\frac{N}{2}}\left(\varepsilon^{\alpha}+(|V(y)-V(x_0)|)\right)\|\varphi\|_{\varepsilon}
\end{equation*}
for $\varphi \in H_{\varepsilon}$. Here $\alpha$ denotes the order of the H\"{o}lder continuity of $V$ in $B_{r_0}(x_0)$.
\end{Lem}
\par
Next we give estimates for $R_{\varepsilon}$ and its derivatives $R_{\varepsilon}^{(i)}$ for $i=1,2$.
\begin{Lem}\label{Lem2.30}
There exists a constant $C>0$, independent of $\varepsilon$ and $b$, such that for $i \in\{0,1,2\}$, there hold
\begin{equation*}
\left\|R_{\varepsilon}^{(i)}(\varphi)\right\| \leq C \varepsilon^{-\frac{N(p-1)}{2}}\|\varphi\|_{\varepsilon}^{p+1-i}+C(b+1) \varepsilon^{-\frac{N}{2}}\left(1+\varepsilon^{-\frac{N}{2}}\|\varphi\|_{\varepsilon}\right)\|\varphi\|_{\varepsilon}^{N-i}
\end{equation*}
for all $\varphi \in H_{\varepsilon}$.
\end{Lem}
\par
Now we will give the energy expansion for the approximate solutions.
\begin{Lem}\label{Lem2.3}
Assume that $V$ satisfies $(V 1)$ and $(V 2)$. Then, for $\varepsilon>0$ sufficiently small, there is a small constant $\tau>0$ and $C>0$ such that,
\begin{equation*}
\begin{aligned}
I_{\varepsilon}\left(U_{\varepsilon, y}\right)=& A \varepsilon^{N}+B \varepsilon^{N}\left(\left(V\left(y\right)-V\left(x_0\right)\right)\right)+O(\varepsilon^{N+\alpha})
\end{aligned}
\end{equation*}
where
\begin{equation*}
A=\frac{1}{2} \int_{\mathbb{R}^{N}}\left(a|(-\Delta)^{\frac{s}{2}} U|^{2}+U^{2}\right)dx+\frac{b}{4}\left(\int_{\mathbb{R}^{N}}|(-\Delta)^{\frac{s}{2}} U|^{2}dx\right)^{2}-\frac{1}{p+1} \int_{\mathbb{R}^{N}} U^{p+1}dx,
\end{equation*}
and
\begin{equation*}
B=\frac{1}{2} \int_{\mathbb{R}^{N}} U^{2}dx.
\end{equation*}
\end{Lem}
\par
Now we complete {\bf Step 1} for the Lyapunov-Schmidt reduction method as before. We first consider the operator $\mathcal{L}_{\varepsilon}$,
\begin{equation*}
\begin{aligned}
\left\langle\mathcal{L}_{\varepsilon} \varphi, \psi\right\rangle&=\langle\varphi, \psi\rangle_{\varepsilon}+\varepsilon^{4s-N} b\int_{\mathbb{R}^{N}}\left| (-\Delta)^{\frac{s}{2}} U_{\varepsilon, y}\right|^{2}dx \int_{\mathbb{R}^{N}} (-\Delta)^{\frac{s}{2}} \varphi \cdot (-\Delta)^{\frac{s}{2}} \psi dx\\
&\quad+2 \varepsilon^{4s-N} b\left(\int_{\mathbb{R}^{N}} (-\Delta)^{\frac{s}{2}} U_{\varepsilon, y} \cdot (-\Delta)^{\frac{s}{2}} \varphi dx \right)\left(\int_{\mathbb{R}^{N}} (-\Delta)^{\frac{s}{2}} U_{\varepsilon, y} \cdot (-\Delta)^{\frac{s}{2}} \psi dx\right)-p \int_{\mathbb{R}^{N}} U_{\varepsilon, y}^{p-1} \varphi \psi dx
\end{aligned}
\end{equation*}
for $\varphi, \psi \in H_{\varepsilon} .$ The following result shows that $\mathcal{L}_{\varepsilon}$ is invertible when restricted on $E_{\varepsilon, y}$

\begin{Lem}\label{Lem2.4}
There exist $\varepsilon_{1}>0, \delta_{1}>0$ and $\rho>0$ sufficiently small, such that for every $\varepsilon \in\left(0, \varepsilon_{1}\right), \delta \in\left(0, \delta_{1}\right)$, there holds
\begin{equation*}
\left\|\mathcal{L}_{\varepsilon} \varphi\right\|_{\varepsilon} \geq \rho\|\varphi\|_{\varepsilon}, \quad \forall \varphi \in E_{\varepsilon, y}
\end{equation*}
uniformly with respect to $y \in B_{\delta}(x_0)$.
\end{Lem}
\par
Lemma \ref{Lem2.4} implies that by restricting on $E_{\varepsilon, y}$, the quadratic form $\mathcal{L}_{\varepsilon}: E_{\varepsilon, y} \rightarrow E_{\varepsilon, y}$ has a bounded inverse, with $\left\|\mathcal{L}_{\varepsilon}^{-1}\right\| \leq \rho^{-1}$ uniformly with respect to $y \in B_{\delta}(x_0)$. This further implies the following reduction map.

\begin{Lem}\label{Lem2.5}
There exist $\varepsilon_{0}>0, \delta_{0}>0$ sufficiently small such that for all $\varepsilon \in\left(0, \varepsilon_{0}\right), \delta \in$ $\left(0, \delta_{0}\right)$, there exists a $C^{1}$ map $\varphi_{\varepsilon}: B_{\delta}(x_0) \rightarrow H_{\varepsilon}$ with $y \mapsto \varphi_{\varepsilon, y} \in E_{\varepsilon, y}$ satisfying
\begin{equation*}
\left\langle\frac{\partial J_{\varepsilon}\left(y, \varphi_{\varepsilon, y}\right)}{\partial \varphi}, \psi\right\rangle_{\varepsilon}=0, \quad \forall \psi \in E_{\varepsilon, y}.
\end{equation*}
Moreover, there exists a constant $C>0$ independent of $\varepsilon$ small enough and $\kappa\in(0,\frac{\alpha}{2})$ such that
\begin{equation*}
\|\varphi_{\varepsilon, y}\|_{\varepsilon} \leq C \varepsilon^{\frac{N}{2}+\alpha-\kappa}+C \varepsilon^{\frac{N}{2}} \left(V\left(y\right)-V\left(x_{0}\right)\right)^{1-\kappa}.
\end{equation*}
\end{Lem}
{\bf Proof of Theorem \ref{Thm2.1}:}
Let $\varepsilon_{0}$ and $\delta_{0}$ be defined as in Lemma \ref{Lem2.5} and let $\varepsilon<\varepsilon_{0}$. Fix $0<$ $\delta<\delta_{0}$. Let $y \mapsto \varphi_{\varepsilon, y}$ for $y \in B_{\delta}(x_0)$ be the map obtained in Lemma \ref{Lem2.5}. As aforementioned in {\bf Step 2}, it is equivalent to find a critical point for the function $j_{\varepsilon}$ defined as in \eqref{eq2.2} by Lemma \ref{Lem2.1}. By the Taylor expansion, we have
\begin{equation*}
j_{\varepsilon}(y)=J\left(y, \varphi_{\varepsilon, y}\right)=I_{\varepsilon}\left(U_{\varepsilon, y}\right)+l_{\varepsilon}\left(\varphi_{\varepsilon, y}\right)+\frac{1}{2}\left\langle\mathcal{L}_{\varepsilon} \varphi_{\varepsilon, y}, \varphi_{\varepsilon, y}\right\rangle+R_{\varepsilon}\left(\varphi_{\varepsilon, y}\right).
\end{equation*}
We analyze the asymptotic behavior of $j_{\varepsilon}$ with respect to $\varepsilon$ first.
\par
By Lemma \ref{Lem2.2}-\ref{Lem2.4}, we have
\begin{equation}\label{eq2.5}
\begin{aligned}
j_{\varepsilon}(y)&=I_{\varepsilon}\left(U_{\varepsilon, y}\right)+O\left(\left\|l_{\varepsilon}\right\|\left\|\varphi_{\varepsilon}\right\|+\left\|\varphi_{\varepsilon}\right\|^{2}\right)\\
&=A \varepsilon^{N}+B \varepsilon^{N}\left( V\left(y\right)-V\left(x_{0}\right)\right)+ O(\varepsilon^{N})\left( \varepsilon^{\alpha-\kappa}+ \left(V\left(y\right)-V\left(x_{0}\right)\right)^{1-\kappa}\right)^2+O(\varepsilon^{N+\alpha}).\\
\end{aligned}
\end{equation}
\par
Now consider the minimizing problem
\begin{equation*}
j_{\varepsilon}\left(y_{\varepsilon}\right) \equiv \inf _{y \in B_{\delta}(x_0)} j_{\varepsilon}(y).
\end{equation*}
Assume that $j_{\varepsilon}$ is achieved by some $y_{\varepsilon}$ in $B_{\delta}(x_0) .$ We will prove that $y_{\varepsilon}$ is an interior point of $B_{\delta}(x_0)$.
\par
To prove the claim, we apply a comparison argument. Let $e \in \mathbb{R}^{N}$ with $|e|=1$ and $\eta>1$. We will choose $\eta$ later. Let $z_{\varepsilon}=\varepsilon^{\eta} e \in B_{\delta}(0)$ for a sufficiently large $\eta>1$. By the above asymptotics formula, we have
\begin{equation*}
\begin{aligned}
j_{\varepsilon}\left(z_{\varepsilon}\right)=& A \varepsilon^{N}+B \varepsilon^{N}\left(V\left(z_{\varepsilon}\right)-V(0)\right)+O\left(\varepsilon^{N+\alpha}\right) \\
&+O\left(\varepsilon^{N}\right)\left(\varepsilon^{\alpha-\kappa}+\left(V\left(z_{\varepsilon}\right)-V(0)\right)^{1-\kappa}\right)^{2}.
\end{aligned}
\end{equation*}
Applying the H\"{o}lder continuity of $V$, we derive that
\begin{equation*}
\begin{aligned}
j_{\varepsilon}\left(z_{\varepsilon}\right)=& A \varepsilon^{N}+O\left(\varepsilon^{N+\alpha \eta}\right)+O\left(\varepsilon^{N+\alpha}\right) \\
&+O\left(\varepsilon^{N}\left(\varepsilon^{2(\alpha-\tau)}+\varepsilon^{2 \eta \alpha(1-\kappa)}\right)\right) \\
=& A \varepsilon^{N}+O\left(\varepsilon^{N+\alpha}\right).
\end{aligned}
\end{equation*}
where $\eta>1$ is chosen to be sufficiently large accordingly. Note that we also used the fact that $\kappa\ll \alpha / 2$. Thus, by using $j\left(y_{\varepsilon}\right) \leq j\left(z_{\varepsilon}\right)$ we deduce
\begin{equation*}
B \varepsilon^{N}\left(V\left(y_{\varepsilon}\right)-V(0)\right)+O\left(\varepsilon^{N}\right)\left(\varepsilon^{\alpha-\kappa}+\left(V\left(y_{\varepsilon}\right)-V(0)\right)^{1-\kappa}\right)^{2} \leq O\left(\varepsilon^{N+\alpha}\right)
\end{equation*}
That is,
\begin{equation}\label{eq2.6}
B\left(V\left(y_{\varepsilon}\right)-V(0)\right)+O(1)\left(\varepsilon^{\alpha-\kappa}+\left(V\left(y_{\varepsilon}\right)-V(0)\right)^{1-\kappa}\right)^{2} \leq O\left(\varepsilon^{\alpha}\right)
\end{equation}
If $y_{\varepsilon} \in \partial B_{\delta}(0)$, then by the assumption $(V_2)$, we have
\begin{equation*}
V\left(y_{\varepsilon}\right)-V(0) \geq c_{0}>0
\end{equation*}
for some constant $0<c_{0} \ll 1$ since $V$ is continuous at $x=0$ and $\delta$ is sufficiently small. Thus, by noting that $B>0$ from Lemma \ref{Lem2.3} and sending $\varepsilon \rightarrow 0$, we infer from \eqref{eq2.6} that
\begin{equation*}
c_{0} \leq 0.
\end{equation*}
We reach a contradiction. This proves the claim. Thus $y_{\varepsilon}$ is a critical point of $j_{\varepsilon}$ in $B_{\delta}(x_0)$. Then the existence of solutions now follows from the claim and Lemma \ref{Lem2.1}.

\section{Concentration behavior of solutions}

First we explore some properties of the solutions derived as in Section 2.
Set
\begin{equation*}
\bar{u}_{\varepsilon}(x)=u_{\varepsilon}\left(x+y_{\varepsilon}\right).
\end{equation*}
Then $\bar{u}_{\varepsilon}>0$ solves
\begin{equation}\label{eq3.2}
\left(a+b \int_{\mathbb{R}^{N}}\left|(-\Delta)^{\frac{s}{2}} \bar{u}_{\varepsilon}\right|^{2}dx\right)(-\Delta )^s \bar{u}_{\varepsilon}+\bar{V}_{\varepsilon}(x) \bar{u}_{\varepsilon}=\bar{u}_{\varepsilon}^{p} \quad \text { in } \mathbb{R}^{N}
\end{equation}
with $\bar{V}_{\varepsilon}(x)=V\left(\varepsilon x+\varepsilon y_{\varepsilon}\right)$.
Then in this section we first show that the solutions concentrate around the minima of V.
\subsection{$L^\infty$-estimate}
In order to study the concentration behavior of the semiclassical solutions obtained above, we first establish the $L^\infty$-estimate. Now we recall the following result for completeness.
\begin{Lem}\label{Lem3.1}
Suppose that $h:\mathbb{R}\to\mathbb{R}$ is convex and Lipschitz continuous with the Lipschitz constant $L$, $h(0)=0$. Then for each $u\in H^s(\mathbb{R}^N)$, $h(u)\in H^s(\mathbb{R}^N)$ and
\begin{equation}\label{eq3.4}
(-\Delta)^sh(u)\leq h^\prime(u)(-\Delta)^su
\end{equation}
in the weak sense.
\end{Lem}
\begin{proof}
 First, we claim that $h(u)\in H^s(\mathbb{R}^N)$ for $u\in H^s(\mathbb{R}^N)$. In fact
\begin{equation*}
\begin{split}
[h(u)]_{\mathcal{D}^{s,2}}&=\big(\iint_{\mathbb{R}^N\times \mathbb{R}^N}\frac{|h(u(x))-h(u(y))|^2}{|x-y|^{N+2s}}dxdy\big)^{\frac{1}{2}}\\
&\leq \big(\iint_{\mathbb{R}^N\times \mathbb{R}^N}\frac{L^2|u(x)-u(y)|^2}{|x-y|^{N+2s}}dxdy\big)^{\frac{1}{2}}\\
&=L [u]_{\mathcal{D}^{s,2}},
\end{split}
\end{equation*}
which implies that $h(u)\in \mathcal{D}^{s,2}(\mathbb{R}^N)$. Moreover,
\begin{equation*}
  \int_{\mathbb{R}^N}|h(u)|^2dx=\int_{\mathbb{R}^N}|h(u)-h(0)|^2dx\leq \int_{\mathbb{R}^N}L^2|u|^2dx<\infty,
\end{equation*}
which yields that $h(u)\in L^2(\mathbb{R}^N)$. Therefore, the claim is true.
\par
Next we show that \eqref{eq3.4} holds. Observe that $h^\prime$ exists a.e. in $\mathbb{R}$ since $h$ is Lipschitz continuous. For $\psi\in C_{0}^{
\infty}(\mathbb{R}^N,\mathbb{R})$ with $\psi\geq0$, combining with the convexity of $h$, there holds
\begin{equation*}
\begin{split}
&~~~\int_{\mathbb{R}^N}(-\Delta)^s (h(u))\psi dx\\
&=-\frac{1}{2}C(s)\int\int_{\mathbb{R}^N\times \mathbb{R}^N}\frac{h(u(x+y))+h(u(x-y))-2h(u(x))}{|y|^{N+2s}}\psi(x) dydx,\\
&=-\frac{1}{2}C(s)\int\int_{\mathbb{R}^N\times \mathbb{R}^N}\frac{h(u(x+y))-h(u(x))+h(u(x-y))-h(u(x))}{|y|^{N+2s}}\psi(x) dydx\\
&\leq-\frac{1}{2}C(s)\int\int_{\mathbb{R}^N\times \mathbb{R}^N}\frac{h^\prime(u(x))(u(x+y)-u(x))+h^\prime(u(x))(u(x-y)-u(x))}{|y|^{N+2s}}\psi(x) dydx\\
&=-\frac{1}{2}C(s)\int\int_{\mathbb{R}^N\times \mathbb{R}^N}\frac{h^\prime(u(x))[u(x+y)+u(x-y)-2u(x)]}{|y|^{N+2s}}\psi(x) dydx\\
&=\int_{\mathbb{R}^N}h^\prime(u)(-\Delta)^s u\psi dx.
\end{split}
\end{equation*}
This completes the proof.
\end{proof}
\begin{Rem}\label{rem3.1}
In fact, from the above arguments, one can see that \eqref{eq3.4} holds for a.e. $x\in \mathbb{R}^N$. Moreover, Lemma \ref{Lem3.1} is true for general dimension $N$.
\end{Rem}
\par
The following uniform $L^\infty$-estimate plays a fundamental role in the study of behavior of the maximum points of the solutions, whose proof is related to the Moser iterative method \cite{Moser1960CPAM}. A similar result for the fractional Schr\"{o}dinger equation can be found  in \cite{MR3617721} or \cite{MR4230968}.

\begin{Lem}\label{Lem3.2}
Let $\varepsilon\rightarrow0^+$ and $u_{\varepsilon}$ be a positive solution of \eqref{eq1.1}. Then up to a subsequence, $\bar{u}_{\varepsilon}(x):=u_{\varepsilon}(\varepsilon x+\varepsilon y_{\varepsilon})$ satisfies that $u_{\varepsilon}\in L^\infty(\mathbb{R}^N)$ and there exists $C > 0$ such that
\begin{equation*}
\|\bar{u}_{\varepsilon}\|_{L^\infty(\mathbb{R}^N)}\leq C.
\end{equation*}
\end{Lem}
\begin{proof}
By standard variational method, we can know that $\{\bar{u}_{\varepsilon}\}$ has a convergent subsequence still denote by $\{\bar{u}_{\varepsilon}\}$. Therefore, there exists some $C>0$ such that
\begin{equation*}
  \|\bar{u}_{\varepsilon}\|_\varepsilon\leq C,
\end{equation*}
and hence
\begin{equation}\label{eq3.5}
  \|\bar{u}_{\varepsilon}\|_{2_s^*}\leq C.
\end{equation}
\par
 Let $T > 0$, we define
\begin{equation}\label{eq3.6}
H(t)=\left\{ \begin{array}{ll}0,~~~~~~~~~ ~~~~~~~~~~~~~~~~~~~~~\text{if}~t\leq 0,\\
t^\beta,~~~~~~~~~~~~~~~~~~~~~~~~~~~~~\text{if}~0<t<T,\\
\beta T^{\beta-1}(t-T)+T^\beta,~~~~~~\text{if}~t\geq T,\\
\end{array} \right.
\end{equation}
with $\beta>1$ to be determined later.
\par
Since $H$ is convex and Lipschitz, by Lemma \ref{Lem3.1},
\begin{equation}\label{eq3.7}
(-\Delta )^sH(u)\leq H^{\prime}(u)(-\Delta )^su.
\end{equation}
Thus, by \eqref{eq3.2}, \eqref{eq3.7} and the Sobolev embedding theorem, we have
\begin{equation}\label{eq3.8}
\aligned
\|H(\bar{u}_{\varepsilon})\|^2_{L^{2^*_s}(\mathbb{R}^N)}&\leq S^{-1}\int_{\mathbb{R}^N}|(-\Delta)^{\frac{s}{2}}H(\bar{u}_{\varepsilon})|^2dx\\
&\leq \frac{1}{Sa}\bigg(a+ b \int\int_{\mathbb{R}^{N}\times\mathbb{R}^{N}}\frac{|\bar{u}_{\varepsilon}(x)-\bar{u}_{\varepsilon}(y)|^2}{|x-y|^{N+2s}}dxdy\bigg) \int_{\mathbb{R}^N}|(-\Delta)^{\frac{s}{2}}H(\bar{u}_{\varepsilon})|^2dx\\
&\leq C\bigg(a+ b \int\int_{\mathbb{R}^{N}\times\mathbb{R}^{N}}\frac{|\bar{u}_{\varepsilon}(x)-\bar{u}_{\varepsilon}(y)|^2}{|x-y|^{N+2s}}dxdy\bigg) \int_{\mathbb{R}^N}H(\bar{u}_{\varepsilon})H^{\prime}(\bar{u}_{\varepsilon})(-\Delta)^s\bar{u}_{\varepsilon}dx\\
&=C\int_{\mathbb{R}^N}H(u)H^{\prime}(u)\frac{-V(\varepsilon_n x+\varepsilon_n \tilde{y}_n)\bar{u}_{\varepsilon}+\bar{u}_{\varepsilon}^p}{a+b[\bar{u}_{\varepsilon}]_{\mathcal{D}^{s,2}}^2}dx\\
&\leq C\int_{\mathbb{R}^N}H(u)H^{\prime}(u)[-V(\varepsilon_n x+\varepsilon_n \tilde{y}_n)\bar{u}_{\varepsilon}+\bar{u}_{\varepsilon}^p]dx\\
&\leq C\int_{\mathbb{R}^N}H(\bar{u}_{\varepsilon})H^{\prime}(\bar{u}_{\varepsilon})(1+|\bar{u}_{\varepsilon}|^{2_s^*-1})dx.
\endaligned
\end{equation}
Using the  fact $H(\bar{u}_{\varepsilon})H^{\prime}(\bar{u}_{\varepsilon})\leq \beta^2 \bar{u}_{\varepsilon}^{2\beta-1}$ and $\bar{u}_{\varepsilon}H^{\prime}(\bar{u}_{\varepsilon})\leq \beta H(\bar{u}_{\varepsilon})$, we obtain
\begin{equation}\label{eq3.9}
\aligned
\|H(\bar{u}_{\varepsilon})\|^2_{L^{2^*_s}(\mathbb{R}^N)}&\leq C\int_{\mathbb{R}^N}\bigg(\beta^2 \bar{u}_{\varepsilon}^{2\beta-1}+\beta|\bar{u}_{\varepsilon}|^{2^*_s-2} H^2(\bar{u}_{\varepsilon})\bigg) dx\\
&\leq C\beta^2 \int_{\mathbb{R}^N}\bigg( \bar{u}_{\varepsilon}^{2\beta-1}+|\bar{u}_{\varepsilon}|^{2^*_s-2} H^2(\bar{u}_{\varepsilon})\bigg) dx,
\endaligned
\end{equation}
where $C$ is independent of $\beta$. Notice that the last integral in  \eqref{eq3.9} is well defined for every $T$ in the definition of $H$.
\par
Now we choose $\beta$ in \eqref{eq3.9} such that $2\beta-1=2_s^*$, and denote it as
$$\beta_1:=\frac{2_s^*+1}{2}.$$
Let $R > 0$ to be fixed later. For the last integral in \eqref{eq3.9}, we apply the H\"{o}lder's inequality with exponents $r=\frac{2_s^*}{2}$ and $r^{\prime}=\frac{2_s^*}{2_s^*-2}$,
\begin{equation}\label{eq3.10}
\aligned
\int_{\mathbb{R}^N}|\bar{u}_{\varepsilon}|^{2^*_s-2} H^2(\bar{u}_{\varepsilon})dx&= \int_{\{\bar{u}_{\varepsilon}\leq R\}}|\bar{u}_{\varepsilon}|^{2^*_s-2} H^2(\bar{u}_{\varepsilon})dx +\int_{\{\bar{u}_{\varepsilon}> R\}}|\bar{u}_{\varepsilon}|^{2^*_s-2} H^2(\bar{u}_{\varepsilon})dx \\
&\leq R^{2_s^*-1} \int_{\{\bar{u}_{\varepsilon}\leq R\}}\frac{ H^2(\bar{u}_{\varepsilon})}{\bar{u}_{\varepsilon}}dx+\bigg(\int_{\{\bar{u}_{\varepsilon}> R\}}|\bar{u}_{\varepsilon}|^{2^*_s}dx\bigg)^{\frac{2_s^*-2}{2_s^*}}\\
&~~\times \bigg(\int_{R^N}|H(\bar{u}_{\varepsilon})|^{2^*_s}dx\bigg)^{\frac{2}{2_s^*}}.
\endaligned
\end{equation}
Since $\{\bar{u}_{\varepsilon}\}$ is bounded in $H_\varepsilon$, we can take $R$ large enough such that
\begin{equation}\label{eq3.11}
\bigg(\int_{\{\bar{u}_{\varepsilon}> R\}}|\bar{u}_{\varepsilon}|^{2^*_s}dx\bigg)^{\frac{2_s^*-2}{2_s^*}}\leq\frac{1}{2C\beta_1^2}.
\end{equation}
From  \eqref{eq3.9}-\eqref{eq3.11} we have
\begin{equation}\label{eq3.12}
\bigg(\int_{\mathbb{R}^N}|H(\bar{u}_{\varepsilon})|^{2^*_s}dx\bigg)^{\frac{2}{2_s^*}}\leq 2C\beta_1^2\bigg(\int_{\mathbb{R}^N}\bar{u}_{\varepsilon}^{2_s^*}dx+ R^{2_s^*-1} \int_{\mathbb{R}^N}\frac{ H^2(\bar{u}_{\varepsilon})}{\bar{u}_{\varepsilon}}dx \bigg).
\end{equation}
Using the fact that $ H(\bar{u}_{\varepsilon})\leq \bar{u}_{\varepsilon}^{\beta_1}$ in the right hand side and taking $T\rightarrow\infty$ we obtain
\begin{equation*}
\bigg(\int_{\mathbb{R}^N}|\bar{u}_{\varepsilon}|^{2^*_s\beta_1}dx\bigg)^{\frac{2}{2_s^*}}\leq 2C\beta_1^2\bigg(\int_{\mathbb{R}^N}\bar{u}_{\varepsilon}^{2_s^*}dx+ R^{2_s^*-1} \int_{\mathbb{R}^N}\bar{u}_{\varepsilon}^{2_s^*}dx \bigg)<\infty,
\end{equation*}
hence
\begin{equation*}
\bar{u}_{\varepsilon}\in L^{2^*_s\beta_1}(\mathbb{R}^N).
\end{equation*}
Together with \eqref{eq3.5}, we have
\begin{equation}\label{eq3.13}
\|\bar{u}_{\varepsilon}\|_{2^*_s\beta_1}\leq C,
\end{equation}
uniformly in $\varepsilon$.
Now we suppose $\beta>\beta_1$. Thus, using that $ H(\bar{u}_{\varepsilon})\leq \bar{u}_{\varepsilon}^{\beta}$ in the right hand side of \eqref{eq3.9} and taking $T\rightarrow \infty $, we get
\begin{equation}\label{eq3.14}
\bigg(\int_{\mathbb{R}^N}|\bar{u}_{\varepsilon}|^{2^*_s\beta}dx\bigg)^{\frac{2}{2_s^*}}\leq C\beta^2\bigg(\int_{\mathbb{R}^N}\bar{u}_{\varepsilon}^{2\beta-1}dx+ \int_{\mathbb{R}^N}\bar{u}_{\varepsilon}^{2\beta+2_s^*-2}dx \bigg)<\infty.
\end{equation}
Set $a_1:=\frac{2^*_s(2^*_s-1)}{2(\beta-1)}$ and $b_1:=2\beta-1-a_1$. Note that $\beta>\beta_1$, we see that $a_1\in(0,~2^*_s)$ and $b_1>0$. Thus, by using Young's inequality with exponents $r=\frac{2_s^*}{a_1}$ and $r^{\prime}=\frac{2_s^*}{2_s^*-a_1}$,, we have
\begin{equation}\label{eq3.15}
\aligned
\int_{\mathbb{R}^N}\bar{u}_{\varepsilon}^{2\beta-1}dx&\leq \frac{a_1}{2_s^*}\int_{\mathbb{R}^N}\bar{u}_{\varepsilon}^{2_s^*}dx+\frac{2_s^*-a_1}{2_s^*}\int_{\mathbb{R}^N}\bar{u}_{\varepsilon}^{\frac{2_s^*b_1}{2_s^*-a_1}}dx\\
&\leq\int_{\mathbb{R}^N}\bar{u}_{\varepsilon}^{2_s^*}dx+\int_{\mathbb{R}^N}\bar{u}_{\varepsilon}^{2\beta+2_s^*-2}dx\\
&\leq C\bigg(1+\int_{\mathbb{R}^N}\bar{u}_{\varepsilon}^{2\beta+2_s^*-2}dx\bigg).
\endaligned
\end{equation}
Combining  \eqref{eq3.14} and  \eqref{eq3.15}, we conclude that
\begin{equation}\label{eq3.16}
\bigg(\int_{\mathbb{R}^N}|\bar{u}_{\varepsilon}|^{2^*_s\beta}dx\bigg)^{\frac{2}{2_s^*}}\leq C\beta^2\bigg(1+\int_{\mathbb{R}^N}\bar{u}_{\varepsilon}^{2\beta+2_s^*-2}dx\bigg),
\end{equation}
where $C$ remains independently of $\beta$. Therefore,
\begin{equation}\label{eq3.17}
\bigg(1+\int_{\mathbb{R}^N}|\bar{u}_{\varepsilon}|^{2^*_s\beta}dx\bigg)^{\frac{1}{2_s^*(\beta-1)}}\leq (C\beta^2)^{\frac{1}{2(\beta-1)}}\bigg(1+\int_{\mathbb{R}^N}\bar{u}_{\varepsilon}^{2\beta+2_s^*-2}dx\bigg)^{\frac{1}{2(\beta-1)}}.
\end{equation}
Repeating this argument we will define a sequence $\beta_i,~i\geq1$ such that
\begin{equation*}
2\beta_{i+1}+2^*_s-2=2^*_s\beta_i.
\end{equation*}
Thus,
\begin{equation*}
\beta_{i+1}-1=\bigg(\frac{2^*_s}{2}\bigg)^i(\beta_1-1).
\end{equation*}
Replacing it in \eqref{eq3.17} one has
\begin{equation}\label{eq3.18}
\bigg(1+\int_{\mathbb{R}^N}|\bar{u}_{\varepsilon}|^{2^*_s\beta_{i+1}}dx\bigg)^{\frac{2}{2_s^*(\beta_{i+1}-1)}}\leq (C\beta_{i+1}^2)^{\frac{1}{2(\beta_{i+1}-1)}}\bigg(1+\int_{\mathbb{R}^N}\bar{u}_{\varepsilon}^{2\beta_{i+1}+2_s^*-2}dx\bigg)^{\frac{1}{2(\beta_{i+1}-1)}}.
\end{equation}
Denoting $C_{i+1}=C\beta^2_{i+1}$ and
\begin{equation*}
K_i:=\bigg(1+\int_{\mathbb{R}^N}\bar{u}_{\varepsilon}^{2\beta_{i}+2_s^*-2}dx\bigg)^{\frac{1}{2(\beta_{i}-1)}}.
\end{equation*}
So we can rewrite \eqref{eq3.18} as
\begin{equation*}
  K_{i+1}\leq C_{i+1}^{\frac{1}{2(\beta_{i+1}-1)}K_i},
\end{equation*}
and hence we conclude that there exists a constant $D>0$ independent of $i$, such that
\begin{equation*}
K_{i+1}\leq\prod_{j=2}^{i+1}C_j^{\frac{1}{2(\beta_{i}-1)}}K_2\leq DK_1.
\end{equation*}
Therefore,
\begin{equation*}
\bar{u}_{\varepsilon}\in L^{\infty}(\mathbb{R}^N),~\forall \varepsilon.
\end{equation*}
Jointly with \eqref{eq3.13},
\begin{equation*}
\|\bar{u}_{\varepsilon}\|_{L^\infty(\mathbb{R}^N)}\leq C,
\end{equation*}
uniformly in $\varepsilon$. This finishes the proof of Lemma \ref{Lem3.2}.
\end{proof}

\subsection{Concentration behavior of solutions}
Now we are ready to give the proof of concentration behavior of solutions obtained in Theorem \ref{Thm1.1}. Let $u_{\varepsilon}$ be a positive solution of  \eqref{eq1.1} obtained in Theorem \ref{Thm1.1}, then $\bar{u}_{\varepsilon}(x):=u_{\varepsilon}(\varepsilon x+\varepsilon y_{\varepsilon})$ is a solution of the problem
\begin{equation}\label{eq3.19}
\left(a+b \int_{\mathbb{R}^{N}}\left|(-\Delta)^{\frac{s}{2}} \bar{u}_{\varepsilon}\right|^{2}dx\right)(-\Delta )^s \bar{u}_{\varepsilon}+\bar{V}_{\varepsilon}(x) \bar{u}_{\varepsilon}=\bar{u}_{\varepsilon}^{p} \quad \text { in } \mathbb{R}^{N},
\end{equation}
with $\{y_\varepsilon\}\subset\mathbb{R}^N$, $\bar{V}_{\varepsilon}(x):=V(\varepsilon x+\varepsilon y_\varepsilon)$.
From \cite[Lemma 4.6]{MR4230968}, we have
\begin{equation*}
\lim\limits_{\varepsilon\rightarrow 0}I_{\varepsilon}(u_{\varepsilon})=\lim\limits_{\varepsilon\rightarrow 0}\mathcal{J}_{\varepsilon}(\bar{u}_{\varepsilon})=c_{V(x_0)},
\end{equation*}
where
\begin{equation*}
\mathcal{J}_{\varepsilon}(\bar{u}_{\varepsilon})=\frac{a}{2}\int_{\mathbb{R}^{N}}|(-\Delta)^{\frac{s}{2}}\bar{u}_{\varepsilon}|^2dx+\frac{1}{2}
\int_{\mathbb{R}^N}\bar{V}_{\varepsilon}(x)\bar{u}_{\varepsilon}^2dx+\frac{b}{4}(\int_{\mathbb{R}^{N}}|(-\Delta)^{\frac{s}{2}}\bar{u}_{\varepsilon}|^2dx)^2-\frac{1}{p+1}\int_{\mathbb{R}^N}\bar{u}_{\varepsilon}^{p+1}dx
\end{equation*}
is the corresponding energy functional of \eqref{eq3.19} and we define the ground energy corresponding to \eqref{eq3.19} by
\begin{equation*}
c_{\varepsilon}:=\inf_{u\in\mathcal{N}_{\varepsilon}}\mathcal{J}_{\varepsilon},
\end{equation*}
with the Nehari manifold associated to $\mathcal{J}_{\varepsilon}$ defined as
\begin{equation*}
\mathcal{N}_{\varepsilon}:=\{u\in H_\varepsilon\backslash\{0\}:\langle\mathcal{J}_{\varepsilon}'(u),u\rangle=0\}.
\end{equation*}
Similar to  \cite[Lemma 4.8]{MR4230968}, after extracting a subsequence, we have
\begin{equation}\label{eq3.20}
  \bar{u}_{\varepsilon}\to u,~\hbox{in}~H^s(\mathbb{R}^N),
\end{equation}
and
\begin{equation}\label{eq3.21}
  \tilde{y}_\varepsilon:=\varepsilon y_\varepsilon\rightarrow y\in\Lambda:=\{y\in\mathbb{R}^N: V(y)=V(x_0)\}.
\end{equation}
{\bf Claim 1:} $u$ is a positive ground state solution to the limit equation
\begin{equation}\label{eq3.22}
\left(a+b \int_{\mathbb{R}^{N}}\left|(-\Delta)^{\frac{s}{2}}u\right|^{2}dx\right)(-\Delta )^su(x)+V(x_0)u=u^p \quad \text{in}\quad \mathbb{R}^{N}.
\end{equation}
\par
In fact, for any $\varphi\in H^s(\mathbb{R}^N)$, observe that
\begin{equation*}
\left(a+b \int_{\mathbb{R}^{N}}\left|(-\Delta)^{\frac{s}{2}} \bar{u}_{\varepsilon}\right|^{2}dx\right)\int_{\mathbb{R}^{N}}(-\Delta )^s\bar{u}_{\varepsilon}\varphi dx+\int_{\mathbb{R}^{N}}\bar{V}_{\varepsilon}(x)\bar{u}_{\varepsilon}\varphi dx=\int_{\mathbb{R}^{N}}\bar{u}_{\varepsilon}^p\varphi dx
\end{equation*}
and by \eqref{eq3.20} \eqref{eq3.21} and  the fact that $V$ is uniformly continuous,
\begin{equation}\label{eq3.23}
\aligned
&\Big{|}\int_{\mathbb{R}^N}\bar{V}_{\varepsilon}(x)\bar{u}_{\varepsilon}\varphi dx-\int_{\mathbb{R}^N}V(x_0)u\varphi dx\Big{|} \leq\int_{\mathbb{R}^N}\Big{|}\bar{V}_{\varepsilon}(x)\big(\bar{u}_{\varepsilon}\varphi -u\varphi\big)\Big{|}dx +\int_{\mathbb{R}^N}\Big{|}\big(\bar{V}_{\varepsilon}(x)-V(x_0)\big)u\varphi\Big{|}dx\to 0.
\endaligned
\end{equation}
By \eqref{eq3.20}, \eqref{eq3.21} and \eqref{eq3.23}, it is easy to check that
\begin{equation*}
\left(a+b \int_{\mathbb{R}^{N}}\left|(-\Delta)^{\frac{s}{2}}u\right|^{2}dx\right)\int_{\mathbb{R}^{N}}(-\Delta )^su\varphi dx+V(x_0)\int_{\mathbb{R}^{N}}u\varphi dx=\int_{\mathbb{R}^{N}}u^p\varphi dx,~\forall \varphi\in H^s(\mathbb{R}^N),
\end{equation*}
which yields that $u$ is a solution to \eqref{eq3.22}.
\par
On the other hand, by Fatou's Lemma, we have
\begin{equation*}
\aligned
c_{V(x_0)}&\leq \frac{a}{2}\int_{\mathbb{R}^{N}}|(-\Delta)^{\frac{s}{2}}u|^2dx+\frac{V(x_0)}{2}
\int_{\mathbb{R}^N}u^2dx+\frac{b}{4}\bigg(\int_{\mathbb{R}^N}|(-\Delta)^{\frac{s}{2}}u|^2dx\bigg)^2-\frac{1}{p+1}\int_{\mathbb{R}^N}u^{p+1}dx\\
&=\frac{a}{4}\int_{\mathbb{R}^{N}}|(-\Delta)^{\frac{s}{2}}u|^2dx+\frac{V(x_0)}{4}
\int_{\mathbb{R}^N}u^2dx\\
&\leq \liminf_{n\to+\infty}\bigg(\frac{a}{4}\int_{\mathbb{R}^{N}}|(-\Delta)^{\frac{s}{2}}\bar{u}_{\varepsilon}|^2dx
+\frac{1}{4}
\int_{\mathbb{R}^N}\bar{V}_{\varepsilon}(x)\bar{u}_{\varepsilon}^2dx\bigg)\\
&= \liminf_{n\to+\infty}(\mathcal{J}_{\varepsilon_n}(\bar{u}_{\varepsilon})-\frac{1}{4}\langle \mathcal{J}_{\varepsilon_n}^\prime(\bar{u}_{\varepsilon}),\bar{u}_{\varepsilon}\rangle)\\
&=c_{V(x_0)}.
\endaligned
\end{equation*}
Thus, $u$ is a positive ground state solution to \eqref{eq3.22}. \\
{\bf Claim 2:} $\bar{u}_{\varepsilon}(x)\rightarrow 0$ as $|x|\rightarrow \infty$, uniformly in $\varepsilon$.
\par
Indeed, we rewrite \eqref{eq3.19} as
\begin{equation*}
(-\Delta)^s\bar{u}_{\varepsilon}+\bar{u}_{\varepsilon}=\Upsilon_{\varepsilon}, ~\text{in}\ \mathbb{R}^N.
\end{equation*}
where
\begin{equation*}
\Upsilon_{\varepsilon}(x)=(a+b\int_{\mathbb{R}^N}|(-\Delta)^{\frac{s}{2}}\bar{u}_{\varepsilon}|^2dx)^{-1}[\bar{u}_{\varepsilon}(x)^p-\bar{V}_{\varepsilon}(x)\bar{u}_{\varepsilon}(x)]+\bar{u}_{\varepsilon}(x).
\end{equation*}
Then we know from \cite{Felmer-Quaas-Tan2012PRSESA} that
\begin{equation*}
\bar{u}_{\varepsilon}=\mathcal{K}*\Upsilon_{\varepsilon}=\int_{\mathbb{R}^N}\mathcal{K}(x-y)\Upsilon_{\varepsilon}(y)dy,
\end{equation*}
where $\mathcal{K}$ is the Bessel kernel. Moreover, $\mathcal{K}$ has the following properties:
\begin{itemize}
  \item $\mathcal{K}$ is positive, radially symmetric and smooth in $\mathbb{R}^N\setminus \{0\}$,
  \item there is $C>0$ such that $\mathcal{K}(x)\leq \frac{C}{|x|^{N+2s}}$,
  \item $\mathcal{K}\in L^q(\mathbb{R}^N),~\forall q\in [1, \frac{N}{(N-2s)})$.
\end{itemize}
By \eqref{eq3.20}, Lemma \ref{Lem3.2} and its proof and the interpolation on the $L^p-$spaces,
\begin{equation}\label{eq3.24}
\bar{u}_{\varepsilon}\rightarrow u, ~\text{in}~L^p(\mathbb{R}^N), ~\forall p\in (2,+\infty).
\end{equation}
Set
\begin{equation*}
\Upsilon(x)=(a+b\int_{\mathbb{R}^N}|(-\Delta)^{\frac{s}{2}}u|^2dx)^{-1}[u(x)^p-V(y)u(x)]+u(x).
\end{equation*}
It follows from \eqref{eq3.20} and \eqref{eq3.24} that
\begin{equation}\label{eq3.25}
  \Upsilon_{\varepsilon}\rightarrow \Upsilon, ~\text{in}~L^p(\mathbb{R}^N),~\forall~p\in (2,+\infty)
\end{equation}
and
\begin{equation}\label{eq3.26}
  \|\Upsilon_{\varepsilon}\|_{L^\infty(\mathbb{R}^N)}\leq C
\end{equation}
for some $C>0$ and all $\varepsilon$.
\par
 From \eqref{eq3.24}-\eqref{eq3.26}, repeating the proof of \cite[Lemma 2.6]{Alves-Miyagaki2016PDE} with small modifications, we conclude that
 \begin{equation*}
   \bar{u}_{\varepsilon}(x)\rightarrow 0~\text{as}~|x|\rightarrow \infty,
 \end{equation*}
uniformly in $\varepsilon$.
\par
{\bf Claim 3:} There exist $C>0$ such that
\begin{equation*}
\bar{u}_{\varepsilon}(x)\leq \frac{C}{1+|x|^{N+2s}},\ \forall\ x\in \mathbb{R}^N.
\end{equation*}
\par
In fact, according to \cite[Lemma 4.2]{Felmer-Quaas-Tan2012PRSESA}, there exists a continuous function $\bar{\omega}$ such that
\begin{equation}\label{eq3.27}
0<\bar{\omega}(x)\leq \frac{C}{1+|x|^{N+2s}},
\end{equation}
and
\begin{equation}\label{eq3.28}
(-\Delta)^s\bar{\omega}+\frac{V(x_0)}{2(a+2b[u]_{\mathcal{D}^{s,2}}^2)}\bar{\omega}=0,\text{in}\ \mathbb{R}^N\setminus B_{\bar{R}}(0)
\end{equation}
for some suitable $\bar{R}>0$.
From \eqref{eq3.20}, $\bar{u}_{\varepsilon}\to u$ in $L^2(\mathbb{R}^N)$, and hence
\begin{equation*}
[\bar{u}_{\varepsilon}]_{\mathcal{D}^{s,2}}\to [u]_{\mathcal{D}^{s,2}}.
\end{equation*}
Since $\bar{u}_{\varepsilon}$ solves \eqref{eq3.19} and  $\bar{u}_{\varepsilon}(x)\rightarrow 0$ as $|x|\rightarrow \infty$ uniformly in $\varepsilon$, then, for some large $R_1>0$, we obtain
\begin{equation}\label{eq3.29}
\begin{split}
(-\Delta)^s\bar{u}_{\varepsilon}+\frac{V(x_0)}{2(a+2b[u]_{\mathcal{D}^{s,2}}^2)}\bar{u}_{\varepsilon}&=\frac{\bar{u}_{\varepsilon}^p-\bar{V}_{\varepsilon}(x)\bar{u}_{\varepsilon}}
{a+b[\bar{u}_{\varepsilon}]_{\mathcal{D}^{s,2}}^2}+\frac{V(x_0)}{2(a+2b[u]_{\mathcal{D}^{s,2}}^2)}\bar{u}_{\varepsilon}\\
&=\frac{\bar{u}_{\varepsilon}^p-\bar{V}_{\varepsilon}(x)\bar{u}_{\varepsilon}+(a+b[\bar{u}_{\varepsilon}]_{\mathcal{D}^{s,2}}^2)\frac{V(x_0)}{2(a+2b[u]_{\mathcal{D}^{s,2}}^2)}\bar{u}_{\varepsilon}}
{a+b[\bar{u}_{\varepsilon}]_{\mathcal{D}^{s,2}}^2}\\
&\leq \frac{\bar{u}_{\varepsilon}^p-\bar{V}_{\varepsilon}(x)\bar{u}_{\varepsilon}+\frac{V(x_0)}{2}\bar{u}_{\varepsilon}}{a+b[\bar{u}_{\varepsilon}]_{\mathcal{D}^{s,2}}^2}\\
&\leq \frac{\bar{u}_{\varepsilon}^p-\frac{V(x_0)}{2}\bar{u}_{\varepsilon}}{a+b[\bar{u}_{\varepsilon}]_{\mathcal{D}^{s,2}}^2}\\
&\leq 0,
\end{split}
\end{equation}
for $x\in \mathbb{R}^N\setminus B_{R_1}(0)$. Now we take $R_2:=\max\{\bar{R},R_1\}$ and set
\begin{equation}\label{eq3.30}
z_{\varepsilon}:=(\alpha+1)\bar{\omega}-\beta \bar{u}_{\varepsilon},
\end{equation}
where $\alpha:=\sup\limits_{n\in \mathbb{N}}\|\bar{u}_{\varepsilon}\|_\infty<\infty$ and $\beta:=\min\limits_{\bar{B}_{R_2}(0)} \bar{\omega}>0$. We next show that $z_{\varepsilon}\geq 0$ in $\mathbb{R}^N$. For this we suppose by contradiction that, there is a sequence $\{x_{\varepsilon}^j\}$ such that
\begin{equation}\label{eq3.31}
\inf_{x\in \mathbb{R}^N} z_{\varepsilon}(x)=\lim_{j\rightarrow \infty} z_{\varepsilon}(x_{\varepsilon}^j)<0.
\end{equation}
Observe that
\begin{equation*}
\lim_{|x|\rightarrow \infty}\bar{\omega}(x)=0.
\end{equation*}
Combining with $\bar{u}_{\varepsilon}(x)\rightarrow 0$ as $|x|\rightarrow \infty$ uniformly in $\varepsilon$, we obtain
\begin{equation*}
\lim\limits_{|x|\rightarrow\infty} z_{\varepsilon}(x)=0,
\end{equation*}
uniformly in $\varepsilon$. Consequently, the sequence $\{x_{\varepsilon}^j\}$ is bounded and therefore, up to a subsequence, we may assume that $x_{\varepsilon}^j\rightarrow x_{\varepsilon}^*$ as $j\rightarrow\infty$ for some $x_{\varepsilon}^*\in \mathbb{R}^N$. Hence \eqref{eq3.31} becomes
\begin{equation}\label{eq3.32}
z_{\varepsilon}(x_{\varepsilon}^*)=\inf_{x\in \mathbb{R}^N} z_{\varepsilon}(x)<0.
\end{equation}
From \eqref{eq3.32}, we have
\begin{equation}\label{eq3.33}
(-\Delta)^sz_{\varepsilon}(x_{\varepsilon}^*)=-\frac{C(s)}{2}\int_{\mathbb{R}^N}\frac{z_{\varepsilon}(x_{\varepsilon}^*+y)+z_{\varepsilon}(x_{\varepsilon}^*-y)-2z_{\varepsilon}(x_{\varepsilon}^*)}{|y|^{N+2s}}dy\leq 0.
\end{equation}
By \eqref{eq3.30}, we get
\begin{equation*}
z_{\varepsilon}(x)\geq \alpha\beta+\bar{\omega}-\alpha\beta>0,~\text{in}~B_{R_2}(0).
\end{equation*}
Therefore, combining this with \eqref{eq3.32}, we see that
\begin{equation}\label{eq3.34}
x_{\varepsilon}^*\in \mathbb{R}^N\setminus B_{R_2}(0).
\end{equation}
From \eqref{eq3.28}-\eqref{eq3.29}, we conclude that
\begin{equation}\label{eq3.35}
(-\Delta)^sz_{\varepsilon}+\frac{V(x_0)}{2(a+2b[u]_{\mathcal{D}^{s,2}}^2)}z_{\varepsilon}\geq 0,~\text{in}~\mathbb{R}^N\setminus B_{R_2}(0).
\end{equation}
Thanks to \eqref{eq3.34}, we can evaluate \eqref{eq3.35} at the point $x_{\varepsilon}^*$, and recall \eqref{eq3.32},\eqref{eq3.33}, we conclude that
\begin{equation*}
0\leq(-\Delta)^sz_{\varepsilon}(x_{\varepsilon}^*)+\frac{V(x_0)}{2(a+2b[u]_{\mathcal{D}^{s,2}}^2)}z_{\varepsilon}(x_{\varepsilon}^*)<0,
\end{equation*}
this is a contradiction, so $z_{\varepsilon}(x)\geq 0$ in $\mathbb{R}^N$. That is to say, $\bar{u}_{\varepsilon}\leq (\alpha+1)\beta^{-1}\bar{\omega}$, which together with \eqref{eq3.27}, implies that
\begin{equation*}
\bar{u}_{\varepsilon}(x)\leq \frac{C}{1+|x|^{N+2s}},\ \forall\ x\in \mathbb{R}^N.
\end{equation*}
\par
Now, we end the proof of concentration behavior of semiclassical solutions. Using \cite[Proposition 2.9]{Silvestre2007CPAM} again, we see that $\bar{u}_{\varepsilon}\in C^{1,\alpha}(\mathbb{R}^N)$ for any $\alpha<2s-1$.
Considering $\eta_\varepsilon$ the global maximum point of $\bar{u}_{\varepsilon}$, by Lemma \ref{Lem3.2} and {\bf Claim 2}, we see that $\eta_\varepsilon\in B_R(x_0)$ for some $R > 0$. Thus, the global maximum point of $u_{\varepsilon}$ given by $z_{\varepsilon}=\eta_\varepsilon+y_\varepsilon$ satisfies $\varepsilon z_{\varepsilon}=\varepsilon \eta_\varepsilon+\varepsilon y_\varepsilon$. Since $\{\eta_\varepsilon\}$ is bounded, it follows that $\varepsilon z_{\varepsilon}\rightarrow y$, thus $V(\varepsilon z_{\varepsilon})\rightarrow V(x_0)$ as $\varepsilon\rightarrow\infty$. Moreover, by {\bf Claim 3}, we have the following decay estimate
\begin{equation*}
\begin{split}
u(\frac{x}{\varepsilon})&=\bar{u}_{\varepsilon}(\frac{x}{\varepsilon}-y_\varepsilon)\\
&\leq \frac{C}{1+|\frac{x}{\varepsilon}-y_\varepsilon|^{N+2s}}\\
&=\frac{C \varepsilon^{N+2s}}{ \varepsilon^{N+2s}+|x-\tilde{y}_\varepsilon|^{N+2s}}.
\end{split}
\end{equation*}
Now setting $\bar{v}_{\varepsilon}(x)=u_{\varepsilon}(\frac{x}{\varepsilon})$ we can easily see that $\bar{v}_{\varepsilon}(x)$ has the desired properties.

\section{ Local Pohoz\v{a}ev identity}
In this section, we derive a local Pohoz\v{a}ev type identity which plays an important role in the proof of Theorem \ref{Thm1.1}.
\begin{Lem}\label{Lem4.1}
Let $u$ be a positive solution of \eqref{eq1.1} obtained as above. Let $\Omega$ be a bounded smooth domain in $\mathbb{R}^N$. Then, for each $i=1,2,\cdots,N$, there hold
\begin{equation}\label{eq4.1}
\begin{aligned}
\int_{\Omega} \frac{\partial V}{\partial x_{i}} u^{2}dx&=\left(\varepsilon^{2} a+\varepsilon^{4s-N} b \int_{\mathbb{R}^{N}}|(-\Delta)^{\frac{s}{2}} u|^{2}dx\right) \int_{\partial \Omega}\left(|(-\Delta)^{\frac{s}{2}} u|^{2} \nu_{i}-2 \frac{\partial u}{\partial \nu} \frac{\partial u}{\partial x_{i}}\right)d\sigma\\
&+\int_{\partial \Omega} V u^{2} \nu_{i}d\sigma-\frac{2}{p+1} \int_{\partial \Omega} u^{p+1} \nu_{i}d\sigma.
\end{aligned}
\end{equation}
Here $\nu=\left(\nu_{1}, \nu_{2},\cdots, \nu_{N}\right)$ is the unit outward normal of $\partial \Omega$.
\end{Lem}
\begin{proof}We use the some ideas in \cite{MR3426103,Gu-Yang}. Indeed, the definition of nonlocal operator cause some techniques developed for local case can not be adapted immediately to nonlocal case. To overcome these difficulties, we will use an approach due to Caffarelli and Silvestre \cite{Caffarelli-Silvestre2007PDE}, that is, we will apply the s-harmonic extension technique to transform a nonlocal problem to a local one.
\par
For this, we will denote by $\mathbb{R}^{N+1}_+:=\mathbb{R}^N\times(0,+\infty)$. Also, for a point $z\in\mathbb{R}^{N+1}_+$, we will use the notation $z=(x,y)$, with $x\in\mathbb{R}^N$ and $y>0$. For any $u\in H^s(\mathbb{R}^N)$, we define that $w=E_s(u)$ is its s-harmonic extension to the upper half-space $\mathbb{R}^{N+1}_+$, if $w$ is a solution of the problem
\begin{equation*}
\begin{cases}
-div(y^{1-2s}\nabla w)=0& \text{in}\ \mathbb{R}^{N+1}_+,\\
w=u& \text{on}\ \mathbb{R}^N\times\{y=0\}.
\end{cases}
\end{equation*}
\par
Moreover, we define the space $X^s(\mathbb{R}^{N+1}_+)$ and $\dot{H}^s(\mathbb{R}^N)$ as the completion of $C^{\infty}_0\overline{(\mathbb{R}^{N+1}_+)}$ and $C^{\infty}_0{(\mathbb{R}^{N})}$  under the norms
\begin{equation*}
\|w\|^2_{X^s}:=\int_{\mathbb{R}^{N+1}_+}\kappa_sy^{1-2s}|\nabla w|^2dxdy,
\end{equation*}
\begin{equation*}
\|w\|^2_{\dot{H}^s}:=\int_{\mathbb{R}^N}|\nabla w|^2dx,
\end{equation*}
where $\kappa_s>0$ is a normalization constant.
\par
Now we may reformulate the nonlocal Kirchhoff equation \eqref{eq1.1} in a local way, that is
\begin{equation}\label{eq4.2}
\begin{cases}
-\bigg(\varepsilon^{2} a+\varepsilon^{4s-N}b\int_{\mathbb{R}^N}|(-\Delta)^{\frac{s}{2}}u|^2dx\bigg)div(y^{1-2s}\nabla w)=0& \text{in}\ \mathbb{R}^{N+1}_+,\\
-\bigg(\varepsilon^{2} a+\varepsilon^{4s-N}b\int_{\mathbb{R}^N}|(-\Delta)^{\frac{s}{2}}u|^2dx\bigg)\kappa_s\frac{\partial w}{\partial \mu}(x,y)=-V(x)w+w^p& \text{on}\ \mathbb{R}^N\times\{y=0\},
\end{cases}
\end{equation}
where
\begin{equation*}
\frac{\partial w}{\partial\mu}=\lim_{y\rightarrow0^+}\frac{\partial w}{\partial y}(x,y)=-\frac{1}{\kappa_s}(-\Delta)^su(x).
\end{equation*}
If $w$ is a solution of \eqref{eq4.2}, then the trace $u(x)=Tr(w)=w(x,0)$ is a solution of \eqref{eq1.1}. The converse is also true. Moreover, the standard argument shows that $w=E_s(u)\in C^2(\mathbb{R}^{N+1}_+)$. Therefore, both formulations are equivalent.
\par
After multiplying equation \eqref{eq4.2} by $\frac{\partial w}{\partial x_{i}}, i=1,2, \ldots, N$ on $\Omega^{+}\subseteq \mathbb{R}^{N+1}_+$, we obtain
\begin{equation}\label{eq4.3}
\begin{aligned}
&-\bigg(\varepsilon^{2} a+\varepsilon^{4s-N}b\int_{\mathbb{R}^N}|(-\Delta)^{\frac{s}{2}}u|^2dx\bigg)\int_{\Omega^{+}(z)}\kappa_s y^{1-2s}\nabla w(z)\frac{\partial w}{\partial x_{i}}dz\\
&=-\int_{\Omega^{+}(z)}V(z)w(z)\frac{\partial w}{\partial x_{i}}dz+\int_{\Omega^{+}(z)}|w(z)|^p\frac{\partial w}{\partial x_{i}}dz.
\end{aligned}
\end{equation}
Note that
\begin{equation*}
\begin{aligned}
\text { LHS of }\eqref{eq4.3}=\bigg(\varepsilon^{2} a+\varepsilon^{4s-N}b\int_{\mathbb{R}^N}|(-\Delta)^{\frac{s}{2}}u|^2dx\bigg)\left(\frac{1}{2} \int_{\partial \Omega^{+}(z)} \kappa_{s} y^{1-2 s}\left|\nabla w\right|^{2} \nu_{i}d\sigma- \int_{\partial\Omega^{+}(z)} \kappa_{s} y^{1-2 s} \frac{\partial w}{\partial \nu} \frac{\partial w}{\partial y_{i}}d\sigma\right).
\end{aligned}
\end{equation*}
On the other hand, by Green's formula, we have
\begin{equation*}
\text { LHS of }\eqref{eq4.3}=-\frac{1}{2} \int_{{\partial \Omega^+(z)}} u^{2}(x) V(z) \nu_{i}(z)d \sigma+\frac{1}{2} \int_{ \Omega^+(z)} u^{2}(z) \frac{\partial V(z)}{\partial x_{i}}dz+\frac{1}{p+1} \int_{{\partial \Omega^+(z)}}|u(z)|^{p+1} \nu_{i}(x) d\sigma .
\end{equation*}
Combining with them, we obtain
\begin{equation*}
\begin{aligned}
&\bigg(\varepsilon^{2} a+\varepsilon^{4s-N}b\int_{\mathbb{R}^N}|(-\Delta)^{\frac{s}{2}}u|^2dx\bigg)\left(\frac{1}{2} \int_{\partial \Omega^{+}(z)} \kappa_{s} y^{1-2 s}\left|\nabla w\right|^{2} \nu_{i}d\sigma- \int_{\partial\Omega^{+}(z)} \kappa_{s} y^{1-2 s} \frac{\partial w}{\partial \nu} \frac{\partial w}{\partial y_{i}}d\sigma\right) \\
=& \frac{1}{2} \int_{\partial \Omega^+(z)}Vu^{2} \nu_{i}d\sigma-\frac{1}{2} \int_{\Omega^+(z)} \frac{\partial V}{\partial x_{i}}u^2d\sigma-\frac{1}{p+1}\int_{\partial \Omega^+(z)}u^{p+1}\nu_id\sigma,
\end{aligned}
\end{equation*}
which means \eqref{eq4.1}.
\end{proof}
\par
Now, let $u_{\varepsilon}=U_{\varepsilon, y_{\varepsilon}}+\varphi_{\varepsilon, y_{\varepsilon}}$ be an arbitrary solution of \eqref{eq1.1} derived as in Section 2. We know $y_{\varepsilon}=o(1)$ as $\varepsilon \rightarrow 0$. We will improve this asymptotics estimate by assuming that $V$ satisfies the additional assumption $(V_3)$, and by means of the above Pohoz\v{a}ev type identity. We first recall some useful estimates.
\begin{Lem}\label{Lem4.2}
Suppose that $V(x)$ satisfies $(V_3)$, then we have
\begin{equation}\label{eq4.4}
\int_{\mathbb{R}^{N}}\left(V(x_0)-V(x)\right) U_{\varepsilon, y_{\varepsilon}}(x) u(x) dx=O\left(\varepsilon^{\frac{N}{2}+m}+\varepsilon^{\frac{N}{2}}\left|y_{\varepsilon}-x_0\right|^{m}\right)\|u\|_{\varepsilon},
\end{equation}
and
\begin{equation*}
\int_{B_{\bar{d}}\left(y_{\varepsilon}\right)} \frac{\partial V(x)}{\partial x_{i}} U_{\varepsilon, y_{\varepsilon}}(x) u(x)dx=O\left(\varepsilon^{\frac{N}{2}+m-1}+\varepsilon^{\frac{N}{2}}\left|y_{\varepsilon}-x_0\right|^{m-1}\right)\|u\|_{\varepsilon},
\end{equation*}
for any $\bar{d} \in(0, \delta]$, where $u(x) \in H_{\varepsilon}$.
\end{Lem}
\begin{proof}
First, from $(V_3)$ and H\"{o}lder's inequality, for a small constant $d$, we have
\begin{equation}\label{eq4.5}
\begin{aligned}
& |\int_{B_{d}\left(y_{\varepsilon}\right)}\left(V\left(x_0\right)-V(x)\right) U_{\varepsilon, y_{\varepsilon}}(x)  u(x)dx | \\
& \leq C \int_{B_{d}\left(y_{\varepsilon}\right)}\left|x-x_0\right|^{m} U_{\varepsilon, y_{\varepsilon}}(x) |u(x)|dx\\
&\leq C\left(\int_{B_{d}\left(y_{\varepsilon}\right)}\left|x-x_0\right|^{2 m} U^2_{\varepsilon, y_{\varepsilon}}(x)  d x\right)^{\frac{1}{2}}\left(\int_{B_{d}\left(y_{\varepsilon}\right)} u^{2}(x)d x\right)^{\frac{1}{2}} \\
&\leq C \varepsilon^{\frac{N}{2}}\left(\varepsilon^{m}+\left|y_{\varepsilon}-x_0\right|^{m}\right)\|u\|_{\varepsilon}.
\end{aligned}
\end{equation}
Also, by the polynomial decay of $U_{\varepsilon, y_{\varepsilon}}(x)$ in $\mathbb{R}^{N} \backslash B_{d}\left(y_{\varepsilon}\right)$, we can deduce that, for any $\gamma>0$,
\begin{equation}\label{eq4.6}
\left|\int_{\mathbb{R}^{N} \backslash B_{d}\left(y_{\varepsilon}\right)}\left(V(x_0)-V(x)\right) U_{\varepsilon, y_{\varepsilon}}(x) u(x) d x\right| \leq C \varepsilon^{\gamma}\|u\|_{\varepsilon}.
\end{equation}
Then, taking suitable $\gamma>0$, from \eqref{eq4.5} and \eqref{eq4.6}, we get \eqref{eq4.4}.
\par
Next, from $(V_3)$ and H\"{o}lder's inequality, for any $\bar{d} \in(0, \delta]$, we have
\begin{equation*}
\begin{aligned}
\left|\int_{B_{\bar{d}}\left(y_{\varepsilon}\right)} \frac{\partial V(x)}{\partial x_{i}} U_{\varepsilon, y_{\varepsilon}}(x) u(x) dx\right| & \leq C \int_{B_{\bar{d}}\left(y_{\varepsilon}\right)}\left|x-x_0\right|^{m-1} U_{\varepsilon, y_{\varepsilon}}(x)|u(x)| d x \\
& \leq C \varepsilon^{\frac{N}{2}}\left(\varepsilon^{m-1}+\left|y_{\varepsilon}-x_0\right|^{m-1}\right)\|u\|_{\varepsilon}.
\end{aligned}
\end{equation*}
\end{proof}

\begin{Lem}\label{Lem4.3}
For any fixed number $l \in \mathbb{N}^{+}$, suppose that $\left\{u_{i}(x)\right\}_{i=1}^{l}$ satisfies
\begin{equation*}
\int_{\mathbb{R}^{N}}\left|u_{i}(x)\right| dx<+\infty, \quad i=1, \ldots, l.
\end{equation*}
Then for any $x_{0}$, there exist a small constant $d$ and another constant $C$ such that
\begin{equation}\label{eq4.7}
\int_{\partial B_{d}\left(x_{0}\right)}\left|u_{i}(x)\right| d \sigma \leq C \int_{\mathbb{R}^{N}}\left|u_{i}(x)\right| d x, \quad \text { for all } i=1, \ldots, l .
\end{equation}
\end{Lem}
\begin{proof}
Let $M_{i}=\int_{\mathbb{R}^{N}}\left|u_{i}(x)\right| dx$, for $i=1, \ldots, l$. Then for a fixed small $r_{0}>0$,
\begin{equation}\label{eq4.8}
\int_{B_{r_{0}}\left(x_{0}\right)}\left(\sum_{i=1}^{l}\left|u_{i}(x)\right|\right) d x \leq \sum_{i=1}^{l} M_{i}, \quad \text { for all } i=1, \ldots, l.
\end{equation}
On the other hand,
\begin{equation}\label{eq4.9}
\int_{B_{r_{0}}\left(x_{0}\right)}\left(\sum_{i=1}^{l}\left|u_{i}(x)\right|\right) d x \geq \int_{0}^{r_{0}} \int_{\partial B_{r}\left(x_{0}\right)}\left(\sum_{i=1}^{l}\left|u_{i}(x)\right|\right) d \sigma d r .
\end{equation}
Then \eqref{eq4.8} and \eqref{eq4.9} imply that there exists a constant $d<r_{0}$ such that
\begin{equation}\label{eq4.10}
\int_{\partial B_{r}\left(x_{0}\right)}\left|u_{i}(x)\right| d \sigma \leq \int_{\partial B_{d}\left(x_{0}\right)}\left(\sum_{i=1}^{l}\left|u_{i}(x)\right|\right) d \sigma \leq \frac{\sum_{i=1}^{l} M_{i}}{r_{0}}, \quad \text { for all } i=1, \ldots, l .
\end{equation}
So taking $C=\max _{1 \leq i \leq l} \frac{\sum_{i=1}^{l} M_{i}}{r_{0} M_{i}}$, we can obtain \eqref{eq4.7} from \eqref{eq4.10}.
\end{proof}
\par
Applying Lemma \ref{Lem4.3} to $\varepsilon^{2s}\left|(-\Delta)^{\frac{s}{2}} \varphi_{\varepsilon}\right|^{2}+\varphi_{\varepsilon}^{2}$, there exists a constant $d=d_{\varepsilon} \in(1,2)$ such that
\begin{equation}\label{eq4.11}
\int_{\partial B_{d}\left(y_{\varepsilon}\right)}\left(\varepsilon^{2s}\left|(-\Delta)^{\frac{s}{2}}  \varphi_{\varepsilon}\right|^{2}+\varphi_{\varepsilon}^{2}\right) d \sigma\leq\left\|\varphi_{\varepsilon}\right\|_{\varepsilon}^{2}
\end{equation}
By an elementary inequality, we have
\begin{equation*}
\int_{\partial B_{d}\left(y_{\varepsilon}\right)}\left|(-\Delta)^{\frac{s}{2}}  u_{\varepsilon}\right|^{2}d \sigma \leq 2 \int_{\partial B_{d}\left(y_{\varepsilon}\right)}\left|(-\Delta)^{\frac{s}{2}}  U_{\varepsilon, y_{\varepsilon}}\right|^{2}d \sigma+2 \int_{\partial B_{d}\left(y_{\varepsilon}\right)}\left|(-\Delta)^{\frac{s}{2}}  \varphi_{\varepsilon}\right|^{2}d \sigma.
\end{equation*}
By the proof in Section 3.2 we can know that there exists a small constant $d_1$, such that for any $\gamma>0$ and $0<d<d_1$, we have
\begin{equation}\label{eq4.12}
U_{\varepsilon, y_{\varepsilon}}+|(-\Delta)^{\frac{s}{2}} U_{\varepsilon, y_{\varepsilon}}| =O(\varepsilon^{\gamma}),\quad \text{for}\quad x\in B_d(x),
\end{equation}
and
\begin{equation}\label{eq4.13}
U_{\varepsilon, y_{\varepsilon}}+|(-\Delta)^{\frac{s}{2}} U_{\varepsilon, y_{\varepsilon}}| =o(\varepsilon^{\gamma}),\quad \text{for}\quad x\in \partial B_d(x).
\end{equation}
Hence, for the constant $d$ chosen as above, we deduce
\begin{equation}\label{eq4.14}
\varepsilon^{2s} \int_{\partial B_{d}\left(y_{\varepsilon}\right)}\left|(-\Delta)^{\frac{s}{2}} u_{\varepsilon}\right|^{2}d \sigma=O\left(\left\|\varphi_{\varepsilon}\right\|_{\varepsilon}^{2}+\varepsilon^{\gamma}\right).
\end{equation}
In particularly, it follows from \eqref{eq4.12} and \eqref{eq4.13} that for any $\gamma>0$, it holds
\begin{equation}\label{eq4.15}
\int_{\mathbb{R}^{N}} U_{\varepsilon, y_{\varepsilon}}^{q_{1}} U_{\varepsilon, y_{\varepsilon}}^{q_{2}}d x=O\left(\varepsilon^{\gamma}\right),
\end{equation}
and
\begin{equation}\label{eq4.16}
\int_{\mathbb{R}^{N}} \varepsilon^{2s}(-\Delta)^{\frac{s}{2}} U_{\varepsilon, y_{\varepsilon}} (-\Delta)^{\frac{s}{2}} U_{\varepsilon, y_{\varepsilon}} dx=O\left(\varepsilon^{\gamma}\right),
\end{equation}
where $q_{1}, q_{2}>0$.
\par
Now we can improve the estimate for the asymptotic behavior of $y_{\varepsilon}$ with respect to $\varepsilon$.
\begin{Lem}\label{Lem4.4}
Assume that $V$ satisfies $(V_1)-(V_3)$. Let $u_{\varepsilon}=U_{\varepsilon, y_{\varepsilon}}+\varphi_{\varepsilon}$ be a solution derived as in Theorem \ref{Thm2.1}. Then
\begin{equation*}
\left|y_{\varepsilon}\right|=o(\varepsilon) \quad \text{as}\quad\varepsilon\rightarrow 0.
\end{equation*}
\end{Lem}
\begin{proof}The proof can be found in \cite{MR3426103,MR4021897} for the case $s=1$, we give the nonlocal version  due to the presence of the nonlocal term $\left(\int_{\mathbb{R}^{N}}|(-\Delta)^{\frac{s}{2}} u|^{2}\right) (-\Delta)^s u$ and the fractional operator.
To analyze the asymptotic behavior of $y_{\varepsilon}$ with respect to $\varepsilon$, we apply the Pohoz\v{a}ev-type identity \eqref{eq4.1} to $u=u_{\varepsilon}$ with $\Omega=B_{d}\left(y_{\varepsilon}\right)$, where $d \in(1,2)$ is chosen as in \eqref{eq4.11}. Note that $d$ is possibly dependent on $\varepsilon$. We get
\begin{equation}\label{eq4.17}
\int_{B_{d}\left(y_{\varepsilon}\right)} \frac{\partial V}{\partial x_{i}}\left(U_{\varepsilon, y_{\varepsilon}}+\varphi_{\varepsilon}\right)^{2}dx=: \sum_{i=1}^{3} I_{i}
\end{equation}
with
\begin{equation*}
I_{1}=\left(\varepsilon^{2s} a+\varepsilon^{4s-N} b \int_{\mathbb{R}^{N}}\left|(-\Delta)^{\frac{s}{2}} u_{\varepsilon}\right|^{2}dx\right) \int_{\partial B_{d}\left(y_{\varepsilon}\right)}\left(\left|(-\Delta)^{\frac{s}{2}} u_{\varepsilon}\right|^{2} \nu_{i}-2 \frac{\partial u_{\varepsilon}}{\partial \nu} \frac{\partial u_{\varepsilon}}{\partial x_{i}}\right)d\sigma,
\end{equation*}
\begin{equation*}
I_{2}=\int_{\partial B_{d}\left(y_{\varepsilon}\right)} V(x) u_{\varepsilon}^{2}(x) \nu_{i}d\sigma,
\end{equation*}
and
\begin{equation*}
I_{3}=-\frac{2}{p+1} \int_{\partial B_{d}\left(y_{\varepsilon}\right)} u_{\varepsilon}^{p+1}(x) \nu_{i}d\sigma.
\end{equation*}
It follows from Theorem \ref{Thm2.1} that
\begin{equation*}
\varepsilon^{2s} a+\varepsilon^{4s-N} b \int_{\mathbb{R}^{N}}\left|(-\Delta)^{\frac{s}{2}} u_{\varepsilon}\right|^{2}dx=O\left(\varepsilon^{2s}\right).
\end{equation*}
Thus, from \eqref{eq4.14} we deduce $I_{1}=O\left(\left\|\varphi_{\varepsilon}\right\|_{\varepsilon}^{2}+\varepsilon^{\gamma}\right)$. Using similar arguments and choosing a suitable $d$ if necessary, we also get $I_{2}=O\left(\left\|\varphi_{\varepsilon}\right\|_{\varepsilon}^{2}+\varepsilon^{\gamma}\right)$. For $I_{3}$, by Lemma \ref{Lem4.3} we have
\begin{equation*}
\begin{aligned}
I_{3}&\leq C \left(\int_{\partial B_{d}\left(y_{\varepsilon}\right)} |\varphi_{\varepsilon}(x)|^{p+1}d\sigma+\varepsilon^{\gamma}\right)\\
&\leq C\left(\int_{\mathbb{R}^N} |\varphi_{\varepsilon}(x)|^{p+1}dx+\varepsilon^{\gamma}\right)\\
&\leq C(\|\varphi_{\varepsilon}\|_{\varepsilon}^2+\varepsilon^{\gamma}).
\end{aligned}
\end{equation*}
 Hence
\begin{equation}\label{eq4.18}
\sum_{i=1}^{3} I_{i}=O\left(\left\|\varphi_{\varepsilon}\right\|_{\varepsilon}^{2}+\varepsilon^{\gamma}\right).
\end{equation}
To estimate the left hand side of \eqref{eq4.17}, notice that from Lemma \ref{Lem4.2}
\begin{equation}\label{eq4.19}
\begin{aligned}
&\int_{B_{d}\left(y_{\varepsilon}\right)} \frac{\partial V(x)}{\partial x_{i}}\left(U_{\varepsilon, y_{\varepsilon}}+\varphi_{\varepsilon}\right)^{2}dx \\
=& \int_{B_{d}\left(y_{\varepsilon}\right)} \frac{\partial V(x)}{\partial x_{i}}\left(U^2_{\varepsilon, y_{\varepsilon}}+\varphi^2_{\varepsilon}\right)d x+2\int_{B_{d}\left(y_{\varepsilon}\right)} \frac{\partial V(x)}{\partial x_{i}} U_{\varepsilon, y_{\varepsilon}} \varphi_{\varepsilon}(x) dx \\
&=\int_{B_{d}\left(y_{\varepsilon}\right)} \frac{\partial V(x)}{\partial x_{i}} U^2_{\varepsilon, y_{\varepsilon}}d x+O\left(\left\|\varphi_{\varepsilon}\right\|_{\varepsilon}^{2}+\varepsilon^{N+2 m-2}+\varepsilon^{N}\left|y_{\varepsilon}\right|^{2 m-2}\right).
\end{aligned}
\end{equation}
By the assumption $(V_3)$, we deduce, for each $i=1,2,\cdots,N$,
\begin{equation*}
\begin{aligned}
\int_{B_{d}\left(y_{\varepsilon}\right)} \frac{\partial V}{\partial x_{i}} U_{\varepsilon, y_{\varepsilon}}^{2}dx &=m c_{i} \int_{B_{d}\left(y_{\varepsilon}\right)}\left|x_{i}\right|^{m-2} x_{i} U_{\varepsilon, y_{\varepsilon}}^{2}dx+O \left(\int_{B_{d}\left(y_{\varepsilon}\right)}|x|^{m} U_{\varepsilon, y_{\varepsilon}}^{2}\right) \\
&=m c_{i} \varepsilon^{N} \int_{B_{\frac{d}{\varepsilon}}(0)}\left|\varepsilon z_{i}+y_{\varepsilon, i}\right|^{m-2}\left(\varepsilon z_{i}+y_{\varepsilon, i}\right) U^{2}(z)dz+O\left(\varepsilon^{N}\left(\varepsilon^{m}+\left|y_{\varepsilon}\right|^{m}\right)\right) \\
&=m c_{i} \varepsilon^{N} \int_{\mathbb{R}^{N}}\left|\varepsilon z_{i}+y_{\varepsilon, i}\right|^{m-2}\left(\varepsilon z_{i}+y_{\varepsilon, i}\right) U^{2}(z)dz+O\left(\varepsilon^{N}\left(\varepsilon^{m}+\left|y_{\varepsilon}\right|^{m}\right)\right).
\end{aligned}
\end{equation*}
Which gives
\begin{equation}\label{eq4.20}
\begin{aligned}
\int_{B_{d}\left(y_{\varepsilon}\right)} \frac{\partial V}{\partial x_{i}}\left(U_{\varepsilon, y_{\varepsilon}}+\varphi_{\varepsilon}\right)^{2}dx=& m c_{i} \varepsilon^{N} \int_{\mathbb{R}^{N}}\left|\varepsilon z_{i}+y_{\varepsilon, i}\right|^{m-2}\left(\varepsilon z_{i}+y_{\varepsilon, i}\right) U^{2}dx \\
&+O\left(\varepsilon^{\frac{N}{2}}\left\|\varphi_{\varepsilon}\right\|_{\varepsilon}+\left\|\varphi_{\varepsilon}\right\|_{\varepsilon}^{2}+\varepsilon^{N}\left(\varepsilon^{m}+\left|y_{\varepsilon}\right|^{m}\right)\right).
\end{aligned}
\end{equation}
Since $c_{i} \neq 0$ by assumption $(V_3)$, combining \eqref{eq4.17}-\eqref{eq4.20} we deduce
\begin{equation*}
\varepsilon^{N} \int_{\mathbb{R}^{N}}\left|\varepsilon z_{i}+y_{\varepsilon, i}\right|^{m-2}\left(\varepsilon z_{i}+y_{\varepsilon, i}\right) U^{2}dx
=O\left(\varepsilon^{\frac{N}{2}}\left\|\varphi_{\varepsilon}\right\|_{\varepsilon}+\left\|\varphi_{\varepsilon}\right\|_{\varepsilon}^{2}+\varepsilon^{N}\left(\varepsilon^{m}+\left|y_{\varepsilon}\right|^{m}\right)\right).
\end{equation*}
By Lemma \ref{Lem2.5} and $(V_3)$,
\begin{equation*}
\left\|\varphi_{\varepsilon}\right\|_{\varepsilon}=O\left(\varepsilon^{\frac{N}{2}}\left(\varepsilon^{m-\tau}+\left|y_{\varepsilon}\right|^{m(1-\tau)}\right)\right).
\end{equation*}
Thus,
\begin{equation}\label{eq4.21}
\int_{\mathbb{R}^{N}}\left|\varepsilon z_{i}+y_{\varepsilon, i}\right|^{m-2}\left(\varepsilon z_{i}+y_{\varepsilon, i}\right) U^{2}dx=O\left(\varepsilon^{m-\tau}+\left|y_{\varepsilon}\right|^{m(1-\tau)}\right).
\end{equation}
On the other hand, let $m^{*}=\min (m, 2)$. We have
\begin{equation}\label{eq4.22}
\begin{aligned}
\left|y_{\varepsilon, i}\right|^{m} \leq &\left|\varepsilon z_{i}+y_{\varepsilon, i}\right|^{m}-m\left|\varepsilon z_{i}+y_{\varepsilon, i}\right|^{m-2}\left(\varepsilon z_{i}+y_{\varepsilon, i}\right)\varepsilon z_{i} \\
&+C\left(\left|\varepsilon z_{i}+y_{\varepsilon, i}\right|^{m-m^{*}}\left|\varepsilon z_{i}\right|^{m^{*}}+\left|\varepsilon z_{i}\right|^{m}\right) \\
\leq & m\left|\varepsilon z_{i}+y_{\varepsilon, i}\right|^{m-2}\left(\varepsilon z_{i}+y_{\varepsilon, i}\right) y_{\varepsilon, i}+C\left(\left|\varepsilon z_{i}\right|^{m}+\left|y_{\varepsilon, i}\right|^{m-m^{*}}\left|\varepsilon z_{i}\right|^{m^{*}}\right)
\end{aligned}
\end{equation}
by the following elementary inequality: for any $e, f \in \mathbb{R}$ and $m>1$, there holds
\begin{equation*}
| | e+f|^{m}-|e|^{m}-m|e|^{m-2} e f |\leq C\left(|e|^{m-m^{*}}|f|^{m^{*}}+|f|^{m}\right)
\end{equation*}
for some $C>0$ depending only on $m .$ So, multiplying \eqref{eq4.22} by $U^{2}$ on both sides and integrate over $\mathbb{R}^{N}$. We get
\begin{equation*}
\left|y_{\varepsilon, i}\right|^{m} \int_{\mathbb{R}^{N}} U^{2}dx \leq m \int_{\mathbb{R}^{N}}\left|\varepsilon z_{i}+y_{\varepsilon, i}\right|^{m-2}\left(\varepsilon z_{i}+y_{\varepsilon, i}\right) y_{\varepsilon, i} U^{2}dx+O\left(\varepsilon^{m}+\left|y_{\varepsilon}\right|^{m-m^{*}} \varepsilon^{m^{*}}\right)
\end{equation*}
for each $i$. Applying  \eqref{eq4.21} to the above estimate yields
\begin{equation*}
\left|y_{\varepsilon}\right|^{m}=O\left(\left(\varepsilon^{m-\tau}+\left|y_{\varepsilon}\right|^{m(1-\tau)}\right)\left|y_{\varepsilon}\right|+\varepsilon^{m}+\left|y_{\varepsilon}\right|^{m-m^{*}} \varepsilon^{m^{*}}\right).
\end{equation*}
Recall that $m \tau<1 .$ Using $\varepsilon$-Young inequality
\begin{equation*}
X Y \leq \delta X^{m}+\delta^{-\frac{m}{m-1}} Y^{\frac{m}{m-1}}, \quad \forall \delta, X, Y>0
\end{equation*}
we deduce
\begin{equation*}
\left|y_{\varepsilon}\right|=O(\varepsilon).
\end{equation*}
\par
We have to prove that $\left|y_{\varepsilon}\right|=o(\varepsilon)$. Assume, on the contrary, that there exist $\varepsilon_{k} \rightarrow 0$ and $y_{\varepsilon_{k}} \rightarrow 0$ such that $y_{\varepsilon_{k}} / \varepsilon_{k} \rightarrow A \in \mathbb{R}^{N}$ with $A=\left(A_{1}, A_{2},\cdots, A_{N}\right) \neq 0$. Then \eqref{eq4.21} gives
\begin{equation*}
\int_{\mathbb{R}^{N}}\left|z_{i}+\frac{y_{\varepsilon_{k}, i}}{\varepsilon_{k}}\right|^{m-2}\left(z_{i}+\frac{y_{\varepsilon_{k}, i}}{\varepsilon_{k}}\right) U^{2}dx=O\left(\varepsilon^{m-\tau}\right).
\end{equation*}
Taking limit in the above gives
\begin{equation*}
\int_{\mathbb{R}^{N}}\left|z_{i}+A_{i}\right|^{m-2}\left(z_{i}+A_{i}\right) U^{2}(z)dz=0.
\end{equation*}
However, since $U=U(|z|)$ is strictly decreasing with respect to $|z|$, we infer that $A=0$. We reach a contradiction. The proof is complete.
\end{proof}
\par
As a consequence of Lemma \ref{Lem4.4} and the assumption $(V_3)$, we infer that
\begin{equation}\label{eq4.23}
\left\|\varphi_{\varepsilon}\right\|_{\varepsilon}=O\left(\varepsilon^{\frac{N}{2}+m(1-\tau)}\right).
\end{equation}
Here we can take $\tau$ so small that $m(1-\tau)>1$ since $m>1$.

\section{Uniqueness of semiclassical bounded states}
In this section we prove the local uniqueness of semiclassical bounded states obtained before. We use a contradiction argument as that of \cite{MR3426103,MR4021897}. Assume $u_{\varepsilon}^{(i)}=U_{\varepsilon, y_{\varepsilon}^{(i)}}+\varphi_{\varepsilon}^{(i)}, i=1,2$, are two distinct solutions derived as in Section 2. By the argument in Section 3, $u_{\varepsilon}^{(i)}$ are bounded functions in $\mathbb{R}^{N}, i=1,2$. Set
\begin{equation*}
\xi_{\varepsilon}=\frac{u_{\varepsilon}^{(1)}-u_{\varepsilon}^{(2)}}{\left\|u_{\varepsilon}^{(1)}-u_{\varepsilon}^{(2)}\right\|_{L^{\infty}\left(\mathbb{R}^{N}\right)}}
\end{equation*}
and set
\begin{equation*}
\bar{\xi}_{\varepsilon}(x)=\xi_{\varepsilon}\left(\varepsilon x+y_{\varepsilon}^{(1)}\right).
\end{equation*}
It is clear that
\begin{equation*}
\left\|\bar{\xi}_{\varepsilon}\right\|_{L^{\infty}\left(\mathbb{R}^{N}\right)}=1.
\end{equation*}
Moreover, by the {\bf Claim 3} in Section 3, there holds
\begin{equation}\label{eq5.1}
\bar{\xi}_{\varepsilon}(x) \rightarrow 0 \quad \text { as }|x| \rightarrow \infty
\end{equation}
uniformly with respect to sufficiently small $\varepsilon>0$. We will reach a contradiction by showing that $\left\|\bar{\xi}_{\varepsilon}\right\|_{L^{\infty}\left(\mathbb{R}^{N}\right)} \rightarrow 0$ as $\varepsilon \rightarrow 0 .$ In view of \eqref{eq5.1}, it suffices to show that for any fixed $R>0$,
\begin{equation}\label{eq5.2}
\left\|\bar{\xi}_{\varepsilon}\right\|_{L^{\infty}\left(B_{R}(0)\right)} \rightarrow 0 \quad \text { as } \varepsilon \rightarrow 0.
\end{equation}
First we have
\begin{Lem}\label{Lem5.1}
 There holds
\begin{equation*}
\left\|\xi_{\varepsilon}\right\|_{\varepsilon}=O\left(\varepsilon^{\frac{N}{2}}\right).
\end{equation*}
\end{Lem}
\begin{proof}
Recall the inner product on $H^{s}(\mathbb{R}^N)$, we can compute that
\begin{equation*}
\begin{aligned}
&\int_{\mathbb{R}^N}|(-\Delta)^{\frac{s}{2}} u|^2-|(-\Delta)^{\frac{s}{2}}v|^2dx\\
&=\int_{\mathbb{R}^N}\left((-\Delta)^{\frac{s}{2}}u+(-\Delta)^{\frac{s}{2}}v \right)\left((-\Delta)^{\frac{s}{2}}u-(-\Delta)^{\frac{s}{2}}v \right)dx\\
&=\int_{\mathbb{R}^N}(-\Delta)^{\frac{s}{2}}(u+v)(-\Delta)^{\frac{s}{2}}(u-v)dx.
\end{aligned}
\end{equation*}
Then assume that $u_{\varepsilon}^{(i)}, i=1,2$, are two solutions to \eqref{eq1.1}, we obtain that
\begin{equation}\label{eq5.3}
\begin{aligned}
&\left(\varepsilon^{2} a+\varepsilon^{4s-N} b \int_{\mathbb{R}^{N}}\left|(-\Delta)^{\frac{s}{2}} u_{\varepsilon}^{(1)}\right|^{2}dx\right) (-\Delta)^s \xi_{\varepsilon}+V \xi_{\varepsilon} \\
&\quad=\varepsilon^{4s-N} b\left(\int_{\mathbb{R}^{N}}(-\Delta)^{\frac{s}{2}}(u_{\varepsilon}^{(1)}+u_{\varepsilon}^{(2)})(-\Delta)^{\frac{s}{2}}\xi_{\varepsilon}dx\right) (-\Delta)^s u_{\varepsilon}^{(2)}+C_{\varepsilon}(x) \xi_{\varepsilon}
\end{aligned}
\end{equation}
and that
\begin{equation}\label{eq5.4}
\begin{aligned}
&\left(\varepsilon^{2} a+\varepsilon^{4s-N} b \int_{\mathbb{R}^{N}}\left|(-\Delta)^{\frac{s}{2}} u_{\varepsilon}^{(2)}\right|^{2}dx\right)(- \Delta)^s \xi_{\varepsilon}+V \xi_{\varepsilon} \\
&\quad=\varepsilon^{4s-N} b\left(\int_{\mathbb{R}^{N}}(-\Delta)^{\frac{s}{2}}(u_{\varepsilon}^{(1)}+u_{\varepsilon}^{(2)})(-\Delta)^{\frac{s}{2}}\xi_{\varepsilon}dx\right) (-\Delta)^s u_{\varepsilon}^{(1)}+C_{\varepsilon}(x) \xi_{\varepsilon}
\end{aligned}
\end{equation}
where
\begin{equation*}
C_{\varepsilon}(x)=p \int_{0}^{1}\left(t u_{\varepsilon}^{(1)}(x)+(1-t) u_{\varepsilon}^{(2)}(x)\right)^{p-1}dt.
\end{equation*}
Adding \eqref{eq5.3} and \eqref{eq5.4} together gives
\begin{equation}\label{eq5.5}
\begin{aligned}
&\left(2 \varepsilon^{2} a+\varepsilon^{4s-N} b \int_{\mathbb{R}^{N}}\left|(-\Delta)^{\frac{s}{2}} u_{\varepsilon}^{(1)}\right|^{2}+\left|(-\Delta)^{\frac{s}{2}} u_{\varepsilon}^{(2)}\right|^{2}dx\right) (-\Delta)^s \xi_{\varepsilon}+2 V \xi_{\varepsilon} \\
&\quad=\varepsilon^{4s-N} b\left(\int_{\mathbb{R}^{N}}(-\Delta)^{\frac{s}{2}}(u_{\varepsilon}^{(1)}+u_{\varepsilon}^{(2)})(-\Delta)^{\frac{s}{2}}\xi_{\varepsilon}dx\right) (-\Delta)^s\left(u_{\varepsilon}^{(1)}+u_{\varepsilon}^{(2)}\right)+2 C_{\varepsilon}(x) \xi_{\varepsilon} .
\end{aligned}
\end{equation}
Multiply $\xi_{\varepsilon}$ on both sides of  \eqref{eq5.5} and integrate over $\mathbb{R}^{N}$. By throwing away the terms containing $b$, we get
\begin{equation*}
\left\|\xi_{\varepsilon}\right\|_{\varepsilon}^{2} \leq \int_{\mathbb{R}^{N}} C_{\varepsilon} \xi_{\varepsilon}^{2}d x.
\end{equation*}
\par
On the other hand, letting $\varphi_{\varepsilon}^{(i)}$ be the small perturbation term corresponding to $u_{\varepsilon}^{(i)}$, then we have
\begin{equation*}
\left|C_{\varepsilon}(x)\right| \leq C\left(U^{p-1}_{\varepsilon, y^{(1)}_{\varepsilon}}+U^{p-1}_{\varepsilon, y^{(2)}_{\varepsilon}}+\left|\varphi_{\varepsilon}^{(1)}(x)\right|^{p-1}+\left|\varphi_{\varepsilon}^{(2)}(x)\right|^{p-1}\right) .
\end{equation*}
Since $\left|\xi_{\varepsilon}(x)\right| \leq 1$, for $i=1,2$, we have
\begin{equation*}
\int_{\mathbb{R}^{N}} U^{p-1}_{\varepsilon, y^{(i)}_{\varepsilon}} \xi_{\varepsilon}^{2}(x) d x \leq C \varepsilon^{N},
\end{equation*}
and
\begin{equation*}
\begin{aligned}
\int_{\mathbb{R}^{N}}\left|\varphi_{\varepsilon}^{(i)}(x)\right|^{p-1} \xi_{\varepsilon}^{2}(x)d x
& \leq C \left(\int_{\mathbb{R}^{N}}\left(\varphi_{\varepsilon}^{(i)}\right)^{2^*_s}\right)^{\frac{p-1}{2^*_s}}\left(\int_{\mathbb{R}^{N}}\left(\xi_{\varepsilon}^{2}\right)^{\frac{2^*_s}{2^*_s+1-p}}\right)^{\frac{2^*_s+1-p}{2^*_s}} dx \\
& \leq C \sum_{i=1}^{2}\left\|(-\Delta)^{\frac{s}{2}} \varphi_{\varepsilon}^{(i)}\right\|_{L^{2}\left(\mathbb{R}^{N}\right)}^{p-1}\left(\int_{\mathbb{R}^{N}}\left(\xi_{\varepsilon}^{2}\right)\right)^{\frac{2^*_s+1-p}{2^*_s}}dx\\
&=O\left(\varepsilon^{\frac{p-1}{2}}\right)\left\|\xi_{\varepsilon}\right\|_{\varepsilon}^{\frac{(2^*_s+1-p)2}{2^*_s}}.
\end{aligned}
\end{equation*}
In the last inequality we have used the fact that $\|\varphi_{\varepsilon}\|_{\varepsilon}=O(\varepsilon^{\frac{N}{2}})$. Therefore,
\begin{equation*}
\left\|\xi_{\varepsilon}\right\|_{\varepsilon}^{2}=O\left(\varepsilon^N+\varepsilon^{\frac{p-1}{2}}\right)\left\|\xi_{\varepsilon}\right\|_{\varepsilon}^{\frac{(2^*_s+1-p)2}{2^*_s}}
\end{equation*}
which implies the desired estimate. The proof is complete.
\end{proof}
\par
Next we study the asymptotic behavior of $\bar{\xi}_{\varepsilon}$.
\begin{Lem}\label{Lem5.2}
Let $\bar{\xi}_{\varepsilon}=\xi_{\varepsilon}\left(\varepsilon x+y_{\varepsilon}^{(1)}\right) .$ There exist $d_{i} \in \mathbb{R}, i=1,2,\cdots,N$, such that (up to a subsequence)
\begin{equation*}
\bar{\xi}_{\varepsilon} \rightarrow \sum_{i=1}^{N} d_{i} \partial_{x_{i}} U \quad \text { in } C_{\operatorname{loc}}^{1}\left(\mathbb{R}^{N}\right)
\end{equation*}
as $\varepsilon\rightarrow0$.
\end{Lem}
\begin{proof}
It is straightforward to deduce from \eqref{eq5.3} that $\bar{\xi}_{\varepsilon}$ solves
\begin{equation}\label{eq5.6}
\begin{aligned}
&\left(a+\varepsilon^{N-4s} b \int_{\mathbb{R}^{N}}\left|(-\Delta)^{\frac{s}{2}} u_{\varepsilon}^{(1)}\right|^{2}dx\right) (-\Delta)^s \bar{\xi}_{\varepsilon}+V\left(\varepsilon x+y_{\varepsilon}^{(1)}\right) \bar{\xi}_{\varepsilon} \\
&\quad=\varepsilon^{N-4s} b\left(\int_{\mathbb{R}^{N}}(-\Delta)^{\frac{s}{2}}(u_{\varepsilon}^{(1)}+u_{\varepsilon}^{(2)})(-\Delta)^{\frac{s}{2}}\xi_{\varepsilon}dx\right)  (-\Delta)^s \left(u_{\varepsilon}^{(2)}\left(\varepsilon x+y_{\varepsilon}^{(1)}\right)\right)+C_{\varepsilon}\left(\varepsilon x+y_{\varepsilon}^{(1)}\right) \bar{\xi}_{\varepsilon}.
\end{aligned}
\end{equation}
For convenience, we introduce
\begin{equation*}
\bar{u}_{\varepsilon}^{(i)}(x)=u_{\varepsilon}^{(i)}\left(\varepsilon x+y_{\varepsilon}^{(1)}\right) \quad \text { and } \quad \bar{\varphi}_{\varepsilon}^{(i)}=\varphi_{\varepsilon}^{(i)}\left(\varepsilon x+y_{\varepsilon}^{(1)}\right)
\end{equation*}
for $i=1,2$. Then, we have
\begin{equation}\label{eq5.7}
\varepsilon^{N-4s} \int_{\mathbb{R}^{N}}\left|(-\Delta)^{\frac{s}{2}} u_{\varepsilon}^{(1)}\right|^{2}dx=\int_{\mathbb{R}^{N}}\left|(-\Delta)^{\frac{s}{2}} \bar{u}_{\varepsilon}^{(1)}\right|^{2}dx
\end{equation}
and
\begin{equation}\label{eq5.8}
\varepsilon^{N-4s} b\left(\int_{\mathbb{R}^{N}} |(-\Delta)^{\frac{s}{2}} u_{\varepsilon}^{(1)}|^2+|(-\Delta)^{\frac{s}{2}}u_{\varepsilon}^{(2)}|^2dx\right)=b\left(\int_{\mathbb{R}^{N}} |(-\Delta)^{\frac{s}{2}} \bar{u}_{\varepsilon}^{(1)}|^2+|(-\Delta)^{\frac{s}{2}}\bar{u}_{\varepsilon}^{(2)}|^2dx\right)
\end{equation}
which are uniformly bounded. Moreover, we have
\begin{equation}\label{eq5.9}
\int_{\mathbb{R}^{N}}\left|(-\Delta)^{\frac{s}{2}} \bar{\varphi}_{\varepsilon}^{(i)}\right|^{2}dx=\varepsilon^{-N} O\left(\left\|\bar{\varphi}_{\varepsilon}^{(i)}\right\|_{\varepsilon}^{2}\right)=O\left(\varepsilon^{2 m(1-\tau)}\right)
\end{equation}
by \eqref{eq4.23}, and
\begin{equation}\label{eq5.10}
\int_{\mathbb{R}^{N}}\left|(-\Delta)^{\frac{s}{2}} \bar{\xi}_{\varepsilon}\right|^{2}dx=\varepsilon^{N-4s} \int_{\mathbb{R}^{N}}|(-\Delta)^{\frac{s}{2}} \xi|^{2}dx=O(1)
\end{equation}
by Lemma \ref{Lem5.1}.
\par
Thus, in view of $\left\|\bar{\xi}_{\varepsilon}\right\|_{L^{\infty}\left(\mathbb{R}^{N}\right)}=1$ and estimates in the above, the elliptic regularity theory implies that $\bar{\xi}_{\varepsilon}$ is locally uniformly bounded with respect to $\varepsilon$ in $C_{\mathrm{loc}}^{1, \beta}\left(\mathbb{R}^{N}\right)$ for some $\beta \in(0,1)$. As a consequence, we assume (up to a subsequence) that
\begin{equation*}
\bar{\xi}_{\varepsilon} \rightarrow \bar{\xi} \quad \text { in } C_{\operatorname{loc}}^{1}\left(\mathbb{R}^{N}\right)
\end{equation*}
We claim that $\bar{\xi} \in \operatorname{Ker} \mathcal{L}$, that is,
\begin{equation}\label{eq5.11}
\left(a+b \int_{\mathbb{R}^{N}}|(-\Delta)^{\frac{s}{2}} U|^{2}dx\right)(-\Delta)^s \bar{\xi}+2 b\left(\int_{\mathbb{R}^{N}} (-\Delta)^{\frac{s}{2}} U \cdot (-\Delta)^{\frac{s}{2}} \bar{\xi}dx\right)(-\Delta)^s  U+\bar{\xi}=p U^{p-1} \bar{\xi},
\end{equation}
which can be seen as the limiting equation of \eqref{eq5.6}. It follows from \eqref{eq5.7} and \eqref{eq5.9} that
\begin{equation}\label{eq5.12}
\begin{aligned}
\varepsilon^{N-4s} b \int_{\mathbb{R}^{N}}\left|(-\Delta)^{\frac{s}{2}} u_{\varepsilon}^{(1)}\right|^{2}dx-b \int_{\mathbb{R}^{N}}|(-\Delta)^{\frac{s}{2}} U|^{2}dx &=b \int_{\mathbb{R}^{N}}\left(\left|(-\Delta)^{\frac{s}{2}} \bar{u}_{\varepsilon}^{(1)}\right|^{2}-|(-\Delta)^{\frac{s}{2}} U|^{2}\right)dx \\
&=b \int_{\mathbb{R}^{N}}\left(\left|(-\Delta)^{\frac{s}{2}} U+(-\Delta)^{\frac{s}{2}} \bar{\varphi}_{\varepsilon}^{(1)}\right|^{2}-|(-\Delta)^{\frac{s}{2}} U|^{2}\right)dx \\
&=O\left(\varepsilon^{m(1-\tau)}\right)\\
&\rightarrow 0.
\end{aligned}
\end{equation}
Similarly, we deduce from  \eqref{eq5.8}-\eqref{eq5.10}that
\begin{equation}\label{eq5.13}
\begin{aligned}
\int_{\mathbb{R}^{N}} (-\Delta)^{\frac{s}{2}}\left(\bar{u}_{\varepsilon}^{(1)}+\bar{u}_{\varepsilon}^{(2)}-2 U\right) \cdot (-\Delta)^{\frac{s}{2}} \bar{\xi}_{\varepsilon}dx=& \int_{\mathbb{R}^{N}} (-\Delta)^{\frac{s}{2}}\left(U\left(x+\left(y_{\varepsilon}^{(1)}-y_{\varepsilon}^{(2)}\right) / \varepsilon\right)-U\right) \cdot (-\Delta)^{\frac{s}{2}} \bar{\xi}_{\varepsilon}dx \\
&+\int_{\mathbb{R}^{N}} (-\Delta)^{\frac{s}{2}}\left(\bar{\varphi}_{\varepsilon}^{(1)}+\bar{\varphi}_{\varepsilon}^{(2)}\right) \cdot (-\Delta)^{\frac{s}{2}} \bar{\xi}_{\varepsilon}dx \\
& \rightarrow 0.
\end{aligned}
\end{equation}
It follows from Lemma \ref{Lem4.4} that, for any $\Phi \in C_{0}^{\infty}\left(\mathbb{R}^{N}\right)$,
\begin{equation*}
\begin{aligned}
\int_{\mathbb{R}^{N}} (-\Delta)^{\frac{s}{2}}\left(\bar{u}_{\varepsilon}^{(2)}-U\right) \cdot (-\Delta)^{\frac{s}{2}} \Phi dx=& \int_{\mathbb{R}^{N}} (-\Delta)^{\frac{s}{2}}\left(U\left(x+\left(y_{\varepsilon}^{(1)}-y_{\varepsilon}^{(2)}\right) / \varepsilon\right)-U\right) \cdot (-\Delta)^{\frac{s}{2}} \Phi dx\\
&+\int_{\mathbb{R}^{N}} (-\Delta)^{\frac{s}{2}} \bar{\varphi}_{\varepsilon}^{(2)} \cdot (-\Delta)^{\frac{s}{2}} \Phi dx\\
&\rightarrow 0.
\end{aligned}
\end{equation*}
Combining the above formulas and $\bar{\xi}_{\varepsilon} \rightarrow \bar{\xi}$ in $C_{\text {loc }}^{1}\left(\mathbb{R}^{N}\right)$, we conclude that
\begin{equation*}
\begin{aligned}
&\frac{b}{\varepsilon}\left(\int_{\mathbb{R}^{N}} (-\Delta)^{\frac{s}{2}}\left(u_{\varepsilon}^{(1)}+u_{\varepsilon}^{(2)}\right) \cdot (-\Delta)^{\frac{s}{2}} \xi_{\varepsilon}dx\right) (-\Delta)^s\left(u_{\varepsilon}^{(2)}\left(\varepsilon x+y_{\varepsilon}^{(1)}\right)\right)\\
& \rightarrow 2 b\left(\int_{\mathbb{R}^{N}} (-\Delta)^{\frac{s}{2}} U \cdot (-\Delta)^{\frac{s}{2}} \bar{\xi}dx\right)(-\Delta)^s U
\end{aligned}
\end{equation*}
in $H^{-s}\left(\mathbb{R}^{N}\right)$.
\par
Now, we estimate $C_{\varepsilon}\left(\varepsilon x+y_{\varepsilon}^{(1)}\right)$
\begin{equation*}
U\left(\frac{x-y_{\varepsilon}^{(1)}}{\varepsilon}\right)-U\left(\frac{x-y_{\varepsilon}^{(2)}}{\varepsilon}\right)=O\left(\frac{y_{\varepsilon}^{(1)}-y_{\varepsilon}^{(2)}}{\varepsilon} \nabla U\left(\frac{x-y_{\varepsilon}^{(1)}+\theta\left(y_{\varepsilon}^{(1)}-y_{\varepsilon}^{(2)}\right)}{\varepsilon}\right)\right),
\end{equation*}
where $0<\theta<1$. Then
\begin{equation*}
u_{\varepsilon}^{(1)}(x)-u_{\varepsilon}^{(2)}(x)=o(1)\nabla U\left(\frac{x-y_{\varepsilon}^{(1)}+\theta\left(y_{\varepsilon}^{(1)}-y_{\varepsilon}^{(2)}\right)}{\varepsilon}\right)+O\left(\left|\varphi_{\varepsilon}^{(1)}(x)\right|+\left|\varphi_{\varepsilon}^{(2)}(x)\right|\right) .
\end{equation*}
So, from \eqref{eq4.12}, for any $\gamma>0$, we have
\begin{equation*}
\begin{aligned}
C_{\varepsilon}(x)=&\left(o(1) \nabla U\left(\frac{x-y_{\varepsilon}^{(1)}+\theta\left(y_{\varepsilon}^{(1)}-y_{\varepsilon}^{(2)}\right)}{\varepsilon}\right)+O\left(\left|\varphi_{\varepsilon}^{(1)}(x)\right|+\left|\varphi_{\varepsilon}^{(2)}(x)\right|\right)\right)^{p-1} \\
&+p U^{p-1}\left(\frac{x-y_{\varepsilon}^{(1)}}{\varepsilon}\right)+o\left(\varepsilon^{\gamma}\right), \quad x \in B_{d}\left(y_{\varepsilon}^{(1)}\right).
\end{aligned}
\end{equation*}
Then using Lemma \ref{Lem4.4}, we can obtain
\begin{equation*}
\begin{aligned}
C_{\varepsilon}\left(\varepsilon x+y_{\varepsilon}^{(1)}\right)=&p U^{p-1}(x)+O\left(\left|\varphi_{\varepsilon}^{(1)}\left(\varepsilon x+y_{\varepsilon}^{(1)}\right)\right|+\left|\varphi_{\varepsilon}^{(2)}\left(\varepsilon x+y_{\varepsilon}^{(1)}\right)\right|\right)^{p-2}+o(1), \quad x \in B_{\frac{d}{\varepsilon}}(0).
\end{aligned}
\end{equation*}
Also, for $i=1,2$, we know
\begin{equation*}
\begin{aligned}
\int_{\mathbb{R}^{N}}\left|\varphi_{\varepsilon}^{(i)}\left(\varepsilon x+y_{\varepsilon}^{(1)}\right)\right|^{p-1}|\Phi(x)| d x & \leq C\left(\int_{\mathbb{R}^{N}}\left|\varphi_{\varepsilon}^{(i)}\left(\varepsilon x+y_{\varepsilon}^{(1)}\right)\right|^{2_s^{*}} dx\right)^{\frac{p-1}{2_s^{*}}}\|\Phi\|_{2^*_s+1-p} \\
& \leq C\left(\varepsilon^{-1}\left\|\varphi_{\varepsilon}^{(i)}\right\|_{\varepsilon}\right)^{p-1} \varepsilon^{-\frac{N(p-1)}{2_s^{*}}}\|\Phi\|_{\frac{2^*_s}{2_s^{*}+1-p}}\\
& \leq C\left(\varepsilon^{-\frac{N}{2}}\left\|\varphi_{\varepsilon}^{(i)}\right\|_{\varepsilon}\right)^{p-1}\|\Phi\|_{H^{s}\left(\mathbb{R}^{N}\right)}
\end{aligned}
\end{equation*}
Then for any $\Phi \in C_{0}^{\infty}\left(\mathbb{R}^{N}\right.$ ),
\begin{equation*}
\int_{\mathbb{R}^{N}} C_{\varepsilon}\left(\varepsilon x+y_{\varepsilon}^{(1)}\right) \bar{\xi}_{\varepsilon} \Phi-p \int_{\mathbb{R}^{N}} U^{p-1} \bar{\xi}_{\varepsilon} \Phi=o(1)\|\Phi\|_{H^{s}\left(\mathbb{R}^{N}\right)}.
\end{equation*}
Therefore, we obtain \eqref{eq5.11}. Then $\bar{\xi}=\sum\limits_{i=1}^{N} d_{i} \partial_{x_{i}} U$ follows from Proposition \ref{Pro1.2} for some $d_{i} \in \mathbb{R}(i=1,2,\cdots,N)$, and thus the Lemma is proved.
\end{proof}
\par
Now we prove \eqref{eq5.2} by showing the following lemma.
\begin{Lem}\label{Lem5.3}
Let $d_{i}$ be defined as in Lemma \ref{Lem5.2}. Then
\begin{equation*}
d_{i}=0 \quad \text { for } i=1,2,\cdots,N.
\end{equation*}
\end{Lem}
\begin{proof}
We use the Pohoz\v{a}ev-type identity \eqref{eq4.1} to prove this lemma. Apply \eqref{eq4.1} to $u_{\varepsilon}^{(1)}$ and $u_{\varepsilon}^{(2)}$ with $\Omega=B_{d}\left(y_{\varepsilon}^{(1)}\right)$, where $d$ is chosen in the same way as that of Lemma \ref{Lem5.2}. We obtain
\begin{equation*}
\begin{aligned}
&\int_{B_{d}\left(y_{\varepsilon}^{(1)}\right)} \frac{\partial V}{\partial x_{i}}\left(\left(u_{\varepsilon}^{(1)}\right)^{2}-\left(u_{\varepsilon}^{(2)}\right)^{2}\right)dx \\
&=\left(\varepsilon^{2} a+\varepsilon^{4s-N} b \int_{\mathbb{R}^{N}}\left|(-\Delta)^{\frac{s}{2}} u_{\varepsilon}^{(1)}\right|^{2}dx\right)\int_{\partial B_{d}\left(y_{\varepsilon}^{(1)}\right)}\left(\left|(-\Delta)^{\frac{s}{2}} u_{\varepsilon}^{(1)}\right|^{2} \nu_{i}-2 \frac{\partial u_{\varepsilon}^{(1)}}{\partial v} \frac{\partial u_{\varepsilon}^{(1)}}{\partial x_{i}}\right)d\sigma \\
&\quad-\left(\varepsilon^{2} a+\varepsilon^{4s-N} b \int_{\mathbb{R}^{N}}\left|(-\Delta)^{\frac{s}{2}} u_{\varepsilon}^{(2)}\right|^{2}dx\right) \int_{\partial B_{d}\left(y_{\varepsilon}^{(1)}\right)}\left(\left|(-\Delta)^{\frac{s}{2}} u_{\varepsilon}^{(2)}\right|^{2} \nu_{i}-2 \frac{\partial u_{\varepsilon}^{(2)}}{\partial v} \frac{\partial u_{\varepsilon}^{(2)}}{\partial x_{i}}\right)d\sigma \\
&+\int_{\partial B_{d}\left(y_{\varepsilon}^{(1)}\right)} V(x)\left(\left(u_{\varepsilon}^{(1)}\right)^{2}-\left(u_{\varepsilon}^{(2)}\right)^{2}\right) \nu_{i}d\sigma \\
&\quad-\frac{2}{p+1} \int_{\partial B_{d}\left(y_{\varepsilon}^{(1)}\right)}\left(\left(u_{\varepsilon}^{(1)}\right)^{p+1}-\left(u_{\varepsilon}^{(2)}\right)^{p+1}\right) \nu_{i}d\sigma.
\end{aligned}
\end{equation*}
In terms of $\xi_{\varepsilon}$, we get
\begin{equation}\label{eq5.14}
\int_{B_{d}\left(y_{\varepsilon}^{(1)}\right)} \frac{\partial V}{\partial x_{i}}\left(u_{\varepsilon}^{(1)}+u_{\varepsilon}^{(2)}\right) \xi_{\varepsilon}dx=\sum_{i=1}^{4} I_{i}
\end{equation}
with
\begin{equation*}
\begin{aligned}
&I_{1}=\left(\varepsilon^{2s} a+\varepsilon^{4s-N} b \int_{\mathbb{R}^{N}}\left|(-\Delta)^{\frac{s}{2}} u_{\varepsilon}^{(1)}\right|^{2}dx\right)\int_{\partial B_{d}\left(y_{\varepsilon}^{(1)}\right)}\left(\left((-\Delta)^{\frac{s}{2}} u_{\varepsilon}^{(1)}+(-\Delta)^{\frac{s}{2}} u_{\varepsilon}^{(2)}\right) \cdot (-\Delta)^{\frac{s}{2}} \xi_{\varepsilon}\right) \nu_{i}d\sigma;\\
&I_{2}=-2\left(\varepsilon^{2s} a+\varepsilon^{4s-N} b \int_{\mathbb{R}^{N}}\left|(-\Delta)^{\frac{s}{2}} u_{\varepsilon}^{(1)}\right|^{2}dx\right) \int_{\partial B_{d}\left(y_{\varepsilon}^{(1)}\right)}\left(\frac{\partial u_{\varepsilon}^{(1)}}{\partial\nu} \frac{\partial u_{\varepsilon}^{(1)}}{\partial x_{i}}-\frac{\partial u_{\varepsilon}^{(2)}}{\partial v} \frac{\partial u_{\varepsilon}^{(2)}}{\partial x_{i}}\right)d\sigma;\\
&I_{3}=\varepsilon^{4s-N} b \int_{\partial B_{d}\left(y_{\varepsilon}^{(1)}\right)}\left(\left|(-\Delta)^{\frac{s}{2}} u_{\varepsilon}^{(2)}\right|^{2} \nu_{i}-2 \frac{\partial u_{\varepsilon}^{(2)}}{\partial v} \frac{\partial u_{\varepsilon}^{(2)}}{\partial x_{i}}\right)d\sigma \int_{\mathbb{R}^{N}}\left((-\Delta)^{\frac{s}{2}} u_{\varepsilon}^{(1)}+(-\Delta)^{\frac{s}{2}} u_{\varepsilon}^{(2)}\right) \cdot (-\Delta)^{\frac{s}{2}} \xi_{\varepsilon}dx;\\
&I_{4}=\int_{\partial B_{d}\left(y_{\varepsilon}^{(1)}\right)} V\left(u_{\varepsilon}^{(1)}+u_{\varepsilon}^{(2)}\right) \xi_{\varepsilon} \nu_{i}d\sigma-2 \int_{\partial B_{d}\left(y_{\varepsilon}^{(1)}\right)} A_{\varepsilon} \xi_{\varepsilon} \nu_{i}d\sigma,
\end{aligned}
\end{equation*}
where $A_{\varepsilon}(x)=\int_{0}^{1}\left(t u_{\varepsilon}^{(1)}(x)+(1-t) u_{\varepsilon}^{(2)}(x)\right)^{p}dt$.
\par
We estimate \eqref{eq5.14} term by term. Note that
\begin{equation*}
\varepsilon^{2s} a+\varepsilon^{4s-N} b \int_{\mathbb{R}^{N}}\left|(-\Delta)^{\frac{s}{2}} u_{\varepsilon}^{(i)}\right|^{2}dx=O\left(\varepsilon^{2s}\right)
\end{equation*}
for each $i=1,2$. Moreover, by similar arguments as that of Lemma \ref{Lem4.4}, we have
\begin{equation*}
\int_{\partial B_{d}\left(y_{\varepsilon}^{(1)}\right)}\left|(-\Delta)^{\frac{s}{2}} u_{\varepsilon}^{(i)}\right|^{2}d\sigma=O\left(\left\|(-\Delta)^{\frac{s}{2}} \varphi_{\varepsilon}^{(i)}\right\|_{L^{2}\left(\mathbb{R}^{N}\right)}^{2}\right).
\end{equation*}
Thus, by \eqref{eq5.2} and Lemma \ref{Lem4.2}-\ref{Lem4.4}, we deduce
\begin{equation*}
\sum_{i=1}^{3} I_{i}=\sum_{i=1}^{2} O\left(\varepsilon^{2}\left\|(-\Delta)^{\frac{s}{2}} \varphi_{\varepsilon}^{(i)}\right\|_{L^{2}\left(\mathbb{R}^{N}\right)}^{2}\right)=O\left(\varepsilon^{N+2 m(1-\tau)}\right)
\end{equation*}
and
\begin{equation*}
\int_{\partial B_{d}\left(y_{\varepsilon}^{(1)}\right)} V(x)\left(u_{\varepsilon}^{(1)}+u_{\varepsilon}^{(2)}\right) \xi_{\varepsilon} \nu_{i}d \sigma=O\left(\varepsilon^{N+m(1-\tau)}\right)
\end{equation*}
and
\begin{equation*}
\int_{\partial B_{d}\left(y_{\varepsilon}^{(1)}\right)} A_{\varepsilon} \xi_{\varepsilon} \nu_{i}d\sigma=O\left(\varepsilon^{(N+m(1-\tau)) p}\right).
\end{equation*}
Hence we conclude that
\begin{equation}\label{eq5.15}
\text { the RHS of }\eqref{eq5.14}=O\left(\varepsilon^{N+m(1-\tau)}\right).
\end{equation}
Next we estimate the left hand side of \eqref{eq5.14}. We have
\begin{equation*}
\begin{aligned}
&\int_{B_{d}\left(y_{\varepsilon}^{(1)}\right)} \frac{\partial V}{\partial x_{i}}\left(u_{\varepsilon}^{(1)}+u_{\varepsilon}^{(2)}\right) \xi_{\varepsilon}(x)dx \\
&=m c_{i} \int_{B_{d}\left(y_{\varepsilon}^{(1)}\right)}\left|x_{i}\right|^{m-2} x_{i}\left(u_{\varepsilon}^{(1)}+u_{\varepsilon}^{(2)}\right) \xi_{\varepsilon}(x)dx+O\left(\int_{B_{d}\left(y_{\varepsilon}^{(1)}\right)}\left|x_{i}\right|^{m}\left(u_{\varepsilon}^{(1)}+u_{\varepsilon}^{(2)}\right) \xi_{\varepsilon}(x)dx\right).
\end{aligned}
\end{equation*}
Observe that
\begin{equation*}
\begin{aligned}
&m c_{i} \int_{B_{d}\left(y_{\varepsilon}^{(1)}\right)}\left|x_{i}\right|^{m-2} x_{i}\left(u_{\varepsilon}^{(1)}+u_{\varepsilon}^{(2)}\right) \xi_{\varepsilon}dx \\
&= m c_{i} \varepsilon^{N} \int_{B_{\frac{d}{\varepsilon}}(0)}\left|\varepsilon y_{i}+y_{\varepsilon, i}^{(1)}\right|^{m-2}\left(\varepsilon y_{i}+y_{\varepsilon, i}^{(1)}\right)\left(U(y)+U\left(y+\frac{y_{\varepsilon}^{(1)}-y_{\varepsilon}^{(2)}}{\varepsilon}\right)\right) \bar{\xi}_{\varepsilon}dx \\
&\quad +m c_{i} \int_{B_{d}\left(y_{\varepsilon}^{(1)}\right)}\left|x_{i}\right|^{m-2} x_{i}\left(\varphi_{\varepsilon}^{(1)}+\varphi_{\varepsilon}^{(2)}\right) \xi_{\varepsilon}dx.
\end{aligned}
\end{equation*}
Since $U$ decays at infinity (see Section 3) and $y_{\varepsilon}^{(i)}=o(\varepsilon)$, using Lemma \ref{Lem5.2} we deduce
\begin{equation}\label{eq5.16}
\begin{aligned}
&m c_{i} \varepsilon^{N} \int_{B_{\frac{d}{\varepsilon}}(0)}\left|\varepsilon y_{i}+y_{\varepsilon, i}^{(1)}\right|^{m-2}\left(\varepsilon y_{i}+y_{\varepsilon, i}^{(1)}\right)\left(U(y)+U\left(y+\frac{y_{\varepsilon}^{(1)}-y_{\varepsilon}^{(2)}}{\varepsilon}\right)\right) \bar{\xi}_{\varepsilon}dx \\
&=2 m c_{i} \varepsilon^{m+2} \sum_{j=1}^{N} d_{j} \int_{\mathbb{R}^{N}}\left|y_{i}\right|^{m-2} y_{i} U(y) \partial_{x j} Udx+o\left(\varepsilon^{m+2}\right) \\
&=D_{i} d_{i} \varepsilon^{m+2}+o\left(\varepsilon^{m+2}\right)
\end{aligned}
\end{equation}
where
\begin{equation*}
D_{i}=2 m c_{i} \int_{\mathbb{R}^{N}}|y|^{m-2} y_{i} U(y) \partial_{x_{i}} U \neq 0.
\end{equation*}
On the other hand, by H\"{o}lder's inequality, \eqref{eq4.23} and Lemma \ref{Lem5.1}, we have
\begin{equation}\label{eq5.17}
\begin{aligned}
m c_{i} \int_{B_{d}\left(y_{\varepsilon}^{(1)}\right)}\left|x_{i}\right|^{m-2} x_{i}\left(\varphi_{\varepsilon}^{(1)}+\varphi_{\varepsilon}^{(2)}\right) \xi_{\varepsilon}dx &=\sum_{i=1}^{2} O\left(\int_{\mathbb{R}^{N}}\left|\varphi_{\varepsilon}^{(i)} \| \xi_{\varepsilon}\right|\right) \\
&=\sum_{i=1}^{2} O\left(\left\|\varphi_{\varepsilon}^{(i)}\right\|_{\varepsilon}\left\|\xi_{\varepsilon}\right\|_{\varepsilon}\right) \\
&=O\left(\varepsilon^{N+m(1-\tau)}\right).
\end{aligned}
\end{equation}
Therefore, from  \eqref{eq5.16} and  \eqref{eq5.17}, we deduce
\begin{equation}\label{eq5.18}
m c_{i} \int_{B_{d}\left(y_{\varepsilon}^{(1)}\right)}\left|x_{i}\right|^{m-2} x_{i}\left(u_{\varepsilon}^{(1)}+u_{\varepsilon}^{(2)}\right) \xi_{\varepsilon} d x=D_{i} d_{i} \varepsilon^{N+m-1}+o\left(\varepsilon^{N+m-1}\right).
\end{equation}
Similar arguments give
\begin{equation}\label{eq5.19}
O\left(\int_{B_{d}\left(y_{\varepsilon}^{(1)}\right)}\left|x_{i}\right|^{m}\left(u_{\varepsilon}^{(1)}+u_{\varepsilon}^{(2)}\right) \xi_{\varepsilon} d x\right)=O\left(\varepsilon^{N+m}\right).
\end{equation}
Hence, combining  \eqref{eq5.18} and  \eqref{eq5.19}, we obtain
\begin{equation}\label{eq5.20}
\text{the LHS of}\ \eqref{eq5.14}=D_{i} d_{i} \varepsilon^{N+m-1}+o\left(\varepsilon^{N+m-1}\right).
\end{equation}
So \eqref{eq5.15} and  \eqref{eq5.20} imply that
\begin{equation*}
d_i=0.
\end{equation*}
The proof is complete.
\end{proof}

{\bf Proof of Theorem \ref{Thm1.1}$(ii)$:}
If there exist two distinct solutions $u_{\varepsilon}^{(i)}, i=1,2$, then by setting $\xi_{\varepsilon}$ and $\bar{\xi}_{\varepsilon}$ as above, we find that
\begin{equation*}
\left\|\bar{\xi}_{\varepsilon}\right\|_{L^{\infty}\left(\mathbb{R}^{N}\right)}=1
\end{equation*}
by assumption, and that
\begin{equation*}
\left\|\bar{\xi}_{\varepsilon}\right\|_{L^{\infty}\left(\mathbb{R}^{N}\right)}=o(1) \quad \text { as } \varepsilon \rightarrow 0
\end{equation*}
by \eqref{eq5.1} and  \eqref{eq5.2}. We reach a contradiction. The proof is complete.

\section*{Acknowledgments}
The research of Vicen\c{t}iu D. R\u{a}dulescu was supported by a grant of
the Romanian Ministry of Research, Innovation and Digitization, CNCS/CCCDI-UEFISCDI,
project number PCE 137/2021, within PNCDI III.
The research of Zhipeng Yang was supported by the RTG 2491 "Fourier Analysis and Spectral Theory".
\\
{\bf Conflict of Interest}: On behalf of all authors, the corresponding author states that there is no conflict of interest.\\
{\bf Data availability}: The data that support the findings of this study are available within the article.

\bibliographystyle{plain}
\bibliography{yang}

\end{document}